\documentclass{amsart}

\usepackage{enumitem}

\usepackage{amsmath, amssymb}
\usepackage{xspace}

\newcommand{\GL}[1][]{\ensuremath{\mathop{\mathcal{L}_{#1}}}\xspace}
\newcommand{\GR}[1][]{\ensuremath{\mathop{\mathcal{R}_{#1}}}\xspace}
\newcommand{\GD}[1][]{\ensuremath{\mathop{\mathcal{D}_{#1}}}\xspace}
\newcommand{\GH}[1][]{\ensuremath{\mathop{\mathcal{H}_{#1}}}\xspace}
\newcommand{\GJ}[1][]{\ensuremath{\mathop{\mathcal{J}_{#1}}}\xspace}
\newcommand{\GLE}{\GL[\!A]}
\newcommand{\GRE}{\GR[\!A]}
\newcommand{\GDE}{\GD[A]}

\newcommand{\GLQ}{\GL[Q]}
\newcommand{\GRQ}{\GR[Q]}
\newcommand{\GDQ}{\GD[Q]}

\newcommand{\GJQ}{\GJ[Q]}

\newcommand{\GLs}[1][]{\ensuremath{\mathop{\mathcal{L}^*_{#1}}}\xspace}
\newcommand{\GRs}[1][]{\ensuremath{\mathop{\mathcal{R}^*_{#1}}}\xspace}
\newcommand{\GDs}[1][]{\ensuremath{\mathop{\mathcal{D}^*_{#1}}}\xspace}
\newcommand{\GJs}[1][]{\ensuremath{\mathop{\mathcal{J}^*_{#1}}}\xspace}
\newcommand{\GLt}[1][]{\ensuremath{\mathop{\widetilde{\mathcal{L}}_{#1}}}\xspace}
\newcommand{\GRt}[1][]{\ensuremath{\mathop{\widetilde{\mathcal{R}}_{#1}}}\xspace}
\newcommand{\GDt}[1][]{\ensuremath{\mathop{\widetilde{\mathcal{D}}_{#1}}}\xspace}
\newcommand{\GJt}[1][]{\ensuremath{\mathop{\widetilde{\mathcal{J}}_{#1}}}\xspace}
\newcommand{\GLscGRs}{\mathop{\mathcal{L}^*\!\!\circ\!\mathcal{R}^*}}
\newcommand{\GRscGLs}{\mathop{\mathcal{R}^*\!\!\circ\!\mathcal{L}^*}}

\newcommand{\GHs}{\ensuremath{\mathop{\mathcal{H}^*}}\xspace}
\newcommand{\GHt}{\ensuremath{\mathop{\widetilde{\mathcal{H}}}}\xspace}

\newcommand{\TX}[1][X]{\ensuremath{\mathcal{T}_{#1}}\xspace}
\newcommand{\AlgA}{\ensuremath{\mathcal{A}}\xspace}
\newcommand{\AlgB}{\ensuremath{\mathcal{B}}\xspace}
\newcommand{\AlgC}{\ensuremath{\mathcal{C}}\xspace}
\newcommand{\EndA}[1][\AlgA]{\ensuremath{\mathrm{End}(#1)}\xspace}
\newcommand{\TAB}{\ensuremath{T(\AlgA, \AlgB)}\xspace}
\newcommand{\clotX}[1][X]{\ensuremath{\langle #1 \rangle}\xspace}
\newcommand{\ClotX}[1][X]{\ensuremath{\big\langle #1 \big\rangle}\xspace}
\newcommand{\Csts}{\ensuremath{\ClotX[\emptyset]}\xspace}
\newcommand{\nonempty}{\ensuremath{\neq\emptyset}\xspace}
\newcommand{\Qc}{\ensuremath{Q^c}\xspace}
\newcommand{\Tec}{\ensuremath{T_{e^+}^{\,c}}\xspace}

\newcommand{\ob}[1]{\overline{#1}}
\newcommand{\pmap}[1]{\begin{pmatrix} #1 \end{pmatrix}}
\newcommand{\set}[1]{\ensuremath{\left\lbrace #1 \right\rbrace}\xspace}
\newcommand{\im}{\mathrm{im\,}}
\newcommand{\codim}[1][]{\mathrm{codim_{#1}\,}}
\newcommand{\Card}[1]{\lvert #1 \rvert}
\newcommand{\priv}{\setminus}
\newcommand{\id}{\mathrm{id}}
\newcommand{\N}{\ensuremath{\mathbb{N}}}
\theoremstyle{plain}
\newtheorem{theorem}{Theorem}[section]
\newtheorem{lemma}[theorem]{Lemma}
\newtheorem{corollary}[theorem]{Corollary}
\newtheorem{proposition}[theorem]{Proposition}

\theoremstyle{definition}
\newtheorem*{example}{Example}
\newtheorem*{remark}{Remark}

\theoremstyle{remark}
\newtheorem*{oldproof}{Proof}
\renewenvironment{proof}[1][{}]{\begin{oldproof}[#1]}{\end{oldproof}}

\begin{document}

\title[The semigroup of endomorphisms with restricted range]{The semigroup of endomorphisms with restricted range of an independence algebra.}
\thanks{This research has been carried out thanks to a PhD studentship from the department of Mathematics of the University of York.}
\date{May 2022}

\author{Ambroise Grau}
\address{A. Grau, University of York, Department of Mathematics, YO10 5DD York, UK}
\email{ambroise.grau@york.ac.uk}

\begin{abstract}
Since its introduction by Symons, the semigroup of maps with restricted range has been studied in the context of transformations on a set, or of linear maps on a vector space. 
Sets and vector spaces being particular examples of independence algebras, a natural question that arises is whether by taking the semigroup $\TAB$ of all endomorphisms of an independence algebra $\AlgA$ whose image lie in a subalgebra $\AlgB$, one can obtain corresponding results as in the cases of sets and vector spaces.
In this paper, we put under a common framework the research from Sanwong, Sommanee, Sullivan, Mendes-Gonçalves and all their predecessors.
We describe Green's relations as well as the ideals of $\TAB$ following their lead. 
We then take a new direction, completely describing all of the extended Green's relations on $\TAB$.
We make no restriction on the dimension of our algebras as the results in the finite and infinite dimensional cases generally take the same form.
\end{abstract}

\maketitle

\section{Introduction}\label{section intro}
The full transformation monoid $\TX$ of a set $X$, the monoid of linear transformations $\EndA[V]$ on a vector space $V$, as well as their generalisation to the endomorphism monoid $\EndA$ of an independence algebra $\AlgA$ have been the focus of a large amount of research during the last decades.
Following Malcev's work on the automorphism group of $\TX$, Symons \cite{Symons} investigated the automorphism group of the subsemigroup of $\TX$ consisting of all maps with range restricted to $Y\subseteq X$, denoted by $T(X,Y)$.
Later, Nenthein, Youngkhong and Kemprasit started the study of the properties of $T(X,Y)$ which led to similar studies of $T(V,W)$, the semigroup of linear transformations of a vector space $V$ with restricted range in a subspace $W$.
Since both sets and vector spaces lie in the general framework of independence algebras, we use this more generic structure to describe and extend these studies. 
We give full details of independence algebras in Section \ref{section prelim}. 
For the moment, it is enough for the reader to know they behave similarly to sets and vector spaces in terms of endomorphisms.

Let \AlgA be an independence algebra, and $\AlgB$ be a subalgebra of \AlgA. 
We make no global assumptions concerning the cardinalities of \AlgA and \AlgB, which means that these could be either finite or infinite dimensional algebras. 
Define by $\TAB$ the semigroup of all endomorphisms from $\AlgA$ to $\AlgB$, that is, the set
$$\TAB = \{ \alpha\in\EndA \mid \im\alpha\subseteq \AlgB\}.$$
Of course, if $\AlgB=\emptyset$ then $\TAB=\emptyset$ and we will therefore exclude this case.
Otherwise the set $\TAB$ is easily seen to be a subsemigroup of $\EndA$, and if $\AlgB = \AlgA$ it is equal to the monoid $\EndA$. 
In fact, we will prove later in Corollary \ref{TAB almost never isom EndC} that unless $\AlgB=\AlgA$ or $\AlgB$ is a singleton, then $\TAB$ is not isomorphic to the endomorphism monoid $\EndA[\AlgC]$ of any independence algebra $\AlgC$.

In the case where $\AlgA$ is simply a set, Green's relations on $\TAB$ have been studied by Sanwong and Sommanee in \cite{SS08}, while Sullivan \cite{Sul07} carried out the same work when $\AlgA$ is a vector space, following the exhibition of the regular elements in both cases by Nenthein, Youngkhong and Kemprasit in \cite{NYK05,NK07}.
Later, Mendes-Gonçalves and Sullivan \cite{MGS10,Sul07} gave the structure of the ideals of \TAB for both the set and the vector space case.
The sets $\set{\alpha\in\TAB\mid \mathrm{rank}~\alpha< r}$ for any cardinal $r$ are ideals of \TAB.
However, unlike the case for $\EndA$ these ideals are not the only ones present in $\TAB$ in general, nor do they form a chain.
In \cite{MGS10,Sul07}, the authors used the ideal structure when the dimension of $\AlgB$ is at least $3$ in order to construct two ideals that are not comparable under containment. Doing this, they showed a weaker version of Corollary \ref{TAB almost never isom EndC}, namely, that if $\dim\AlgB\geq 3$ then $\TAB$ cannot be isomorphic to $\EndA[\AlgC]$ for any appropriate $\AlgC$ since ideals of the latter  always form a chain.
Even if the proof we give for this corollary does not rely on the ideal structure of $\TAB$, we will give similar examples of ideals not forming a chain to give the reader an idea on this particular behaviour.

After recalling the notion of independence algebras used in this paper, the first sections will closely follow the work of previous authors. 
We will thus directly state their common results and simply enhance their proofs to fit under the general framework of independence algebras.
We start by looking at the regular elements in Section \ref{section on Q}, and derive Green's relations in Section \ref{section on basic Greens} as well as the ideal structure in Section \ref{section on ideals}.
An astute reader can retrace in the aforementioned papers the origin of the ideas used throughout the proofs present in these three sections.

The major part of this article is devoted to giving new results on the semigroup $\TAB$ where $\AlgA$ is an independence algebra and $\AlgB$ is a subalgebra of $\AlgA$.
From Corollary \ref{cor tab not reg in general}, we have that $\TAB$ is not a regular semigroup in general, and in order to further understand its structure, we turn ourselves to the extended Green's relations $\GLs$, $\GRs$, $\GLt$ and $\GRt$.
These equivalence relations were introduced to complement the study of regular semigroups since these only differ from the initial Green's relations in non-regular semigroups.
While working on the structure of semigroups that are not necessarily regular, Fountain \cite{Fountain-abundant} introduced the notion of \emph{abundant semigroups}, later extended further to the concept of \emph{Fountain semigroups} (formerly called semi-abundant semigroups by El-Qallali \cite{ElQallali}).
These are semigroups in which every $\GLs$ and $\GRs$-class [resp. $\GLt$ and $\GRt$] contains an idempotent, emulating the role that the relations $\GL$ and $\GR$ play in a regular semigroup.
It is well-known that the relations $\GL$ and $\GR$ commute.
However, this is not the case for $\GLs$ and $\GRs$ nor $\GLt$ and $\GRt$ and this presents extra difficulties in describing the joins $\GDs=\GLs\vee\GRs$ and $\GDt=\GLt\vee\GRt$.
Additionally, the equivalence relations $\GJs$ and $\GJt$ induced by the principal $*$-ideals and $\sim$-ideals complete the list of studied relations and are shown to coincide with $\GDs$ and $\GDt$.
All of these extended Green's relations are fully characterised on $\TAB$ in Section \ref{section on extended Greens}, and allow us to give a new example of right-abundant semigroups (that is, where all $\GLs$-classes contain an idempotent), and thus right-Fountain semigroups, that are not abundant nor Fountain since not all $\GRs$ and $\GRt$-classes contain an idempotent.

\section{Preliminaries and notation}\label{section prelim}

In order to allow this paper to be almost self-contained, we give here a short overview of the notion of independence algebras that will be necessary.

These were introduced by Gould in \cite{Gould95} in order to account for the similarities between endomorphisms of sets, vector spaces and free group acts, and to allow their study under a more general setting.
Incidentally, it was noted that independence algebras also corresponds to the $v^*$-algebras described by Narkiewicz \cite{Nark62}.

Let $\AlgA = \ClotX[A; F]$ be a (universal) algebra, with $A$ a non-empty set as its \emph{universe} and $F$ its (possibly infinite) set of \emph{fundamental operations} that all have finite arity. 
As a convention, we will always denote the underlying universe of an algebra, say $\mathcal{K}$, by the corresponding non-scripted capital letter, here $K$.
Fundamental operations can be composed as follows: if $f$ is an $n$-ary operation and $g_1, \dots, g_n$ are all $m$-ary operations, then $f\circ(g_1, \dots, g_n)$ is an $m$-ary operation defined by $f\circ(g_1, \dots, g_n)(x) = f(g_1(x), \dots, g_n(x))$ for all $x\in A^m$.
We also call \emph{projections} the maps $\pi_i^n\colon A^n\to A$ for all $1\leq i \leq n$ sending the $n$-tuple $(x_1,\dots,x_n)$ to $x_i$.
With this, for an integer $n\geq 0$, an $n$-ary \emph{term operation $t$} is a map $t\colon A^n\to A$ built by successive compositions of fundamental operations and projections, where a nullary term operation $f()=a$ will be identified with its image $a\in A$ and will be called a \emph{constant}. 
For any $X\subseteq A$, the universe of the \emph{subalgebra generated by $X$} will be denoted by $\clotX[X]$ and consists of all elements $t(a_1,\dots, a_n)$ where $t$ is a term operation and $a_1,\dots,a_n\in X$. 
The subalgebra generated by $\emptyset$ consists of all the elements that can be obtained using nullary term operations, and thus $\Csts=\emptyset$ if and only if $\AlgA$ has no constants.
In order to facilitate readability and to follow classical semigroup notation, terms of our algebras will be operating on the left while functions on our algebras will be operating on the right. 
We will also give precedence to terms in the evaluation, so that fewer parentheses are needed.

A set $X\subseteq A$ is said to be \emph{independent} if for all $x\in X$ we have that $x\notin \ClotX[X\priv\set{x}]$, and it is a \emph{basis} of $\AlgA$ (or $A$) if it is independent and generates $A$.
We say that an algebra $\AlgA$ satisfy the \emph{exchange property} (EP) if all subsets $X$ of $A$ satisfy the following condition:
\begin{center}
for all $a,b\in A$, if $b\in\ClotX[X\cup\set{a}]$ and $b\notin\clotX[X]$, then $a\in\ClotX[X\cup\set{b}]$.
\end{center}
An important point from \cite{Gould95} is that any algebra $\AlgA$ that satisfies (EP) has a basis, and in addition in such algebra, a set $X$ is a basis if and only if it is a maximal independent set, if and only if it is a minimal generating set.
Additionally, in such an algebra, any independent set can be extended to a basis of $\AlgA$ and all bases have the same cardinality, called the \emph{dimension} (or the \emph{rank}) of $\AlgA$, denoted by $\dim\AlgA$.
Furthermore, for any $B\subseteq A$, the \emph{dimension} (or \emph{rank}) of $B$ is the cardinality of a basis $X$ of the subalgebra $\AlgB=\clotX[B]$ and the \emph{codimension} (or \emph{corank}) of $\AlgB$ in $\AlgA$, denoted $\codim[A] B$, is the cardinality of an independent set $Z$ such that $X\sqcup Z$ forms a basis of $\AlgA$.

For $B,C\subseteq A$, a function $\alpha\colon B \to C$ is a \emph{homomorphism of \AlgA} if for all $n$-ary term operations $t$ and all $b_1,\dots,b_n\in B$, we have that $t(b_1,\dots,b_n)\alpha = t(b_1\alpha, \dots, b_n\alpha)$.
Moreover, $\alpha$ is an \emph{isomorphism} between the subalgebras $\AlgB=\clotX[B]$ and $\AlgC=\clotX[C]$ if it is also bijective, which will be denoted by $\AlgB \cong\AlgC$.
Note that in particular, this forces any homomorphism to act as the identity on constants.
An \emph{endomorphism of $\AlgA$} is a homomorphism $\alpha\colon A\to A$, and the set of all endomorphisms of $\AlgA$ is denoted by $\EndA$, which is a monoid under composition of functions.

An algebra $\AlgA$ satisfying (EP) has the \emph{free basis property} (F) if any map defined from a basis $X\subseteq A$ to $A$ can be extended to an endomorphism of $\AlgA$.
An \emph{independence algebra} is an algebra $\AlgA$ that satisfies both the exchange property and the free basis property.
Consequently, for $\AlgA$ an independence algebra, any subalgebra of $\AlgA$ is an independence algebra in its own right. 
Lastly, for an endomorphism $\alpha\in\EndA$, the \emph{rank of $\alpha$} is defined as the rank of the subalgebra $A\alpha$.
Independence algebras have been classified by Urbanik \cite{Urbanik} and include, as previously mentioned, sets and vector spaces.

We let $\AlgA$ be an independence algebra with a non-empty subalgebra $\AlgB$.
It is clear that $\TAB$, as defined in Section \ref{section intro}, is a subsemigroup of $\EndA$, but not usually a monoid as given by the following lemma:

\begin{lemma}\label{TAB almost never monoid}
The semigroup $\TAB$ is a monoid if and only if $\AlgB = \AlgA$, or $\AlgB$ is a singleton.
\end{lemma}
\begin{proof}
Clearly if $\AlgB = \AlgA$ then $\TAB=\EndA$, and if $\AlgB$ is a singleton, say $\set{b}$, then $\TAB = \set{c_b}$, where $c_b$ is the constant map with value $b$. 
Moreover, both of these are indeed monoids.

Let us now assume that $\AlgB\subsetneq \AlgA$ and let $b_1, b_2\in B$. 
Since $\AlgA$ satisfies (EP) we extend a basis of $B$ to a basis $X$ of $A$ and take $a$ to be one of the basis element in $A\priv B$. 
Define two maps $\alpha,\beta\colon X \to A$ such that $a\alpha= b_1$, $a\beta=b_2$ and $\alpha$, $\beta$ act as identity maps on $X\cap B$ (with elements of $X\priv (B\cup\set{a})$ arbitrarily sent to some elements in $B$).
From the free basis property we know that $\alpha$ and $\beta$ can be extended to endomorphisms of $\AlgA$, and thus we can assume that $\alpha$ and $\beta$ belonged to $\TAB$ in the first place.
Suppose that $\TAB$ is a monoid, and denote its identity by $\epsilon$.
Now let $c\in B$ be such that $a\epsilon=c$. 
Then, as $\epsilon$ is a left identity for $\alpha$ and $\alpha\lvert_B$ is the identity on $B$, we need $b_1 = a\alpha = a\epsilon\alpha = c\alpha = c$.
Similarly, $\epsilon$ is a left identity for $\beta$ and thus $b_2 = a\beta = a\epsilon\beta = c$. Therefore $b_1=b_2$, which means that $\AlgB$ is a singleton.
This shows that for any proper subalgebra $\AlgB$ with at least two distinct elements, the semigroup $\TAB$ does not contain a two-sided identity, and thus is not a monoid.\qed
\end{proof}

From this as well as the fact that the trivial monoid is the endomorphism monoid of an algebra with a single element, we directly obtain the following result:
\begin{corollary}\label{TAB almost never isom EndC}
The semigroup $\TAB$ is isomorphic to $\EndA[\AlgC]$ for $\AlgC$ an independence algebra if and only if $\AlgB = \AlgA$, or $\AlgB$ is a singleton.
\end{corollary}

In order to facilitate the reading of the proofs, the following notation will be used throughout the paper:
\begin{itemize}
    \item the image of $\alpha\in\EndA$ is $\im\alpha = A\alpha$ and its rank is $\rho(\alpha)= \dim(\im\alpha)$;
    \item the set $\{e_i\mid i\in I\}\subseteq A$ is abbreviated as $\{e_i\}$ without necessarily specifying the index set that will be denoted by capitalising the letter used as index, here $I$, and such a set will most of the time be non-empty and can be either finite or infinite;
    \item the closure operator is denoted by $\clotX[\cdot]$ and by writing $C=\clotX[b_1, b_2]$ we mean that $C$ is the subalgebra generated by $b_1$ and $b_2$, while having $C=\clotX[\set{b_1, b_2}]$ gives the additional information that $b_1$ and $b_2$ form a basis for $C$;
    \item similarly, by $B=\clotX[\set{e_i}\sqcup\set{f_j}]$ we mean that the sets $\{e_i\}$ and $\set{f_j}$ are independent sets and that together they form a basis for the subalgebra $\AlgB$, and by $B=C\sqcup\clotX[\set{x_k}]$ we mean that $C$ is a subalgebra of $B$ and that a basis of $C$ can be extended to a basis of $B$ through $\set{x_k}$;
    \item for $A = \ClotX[\{x_i\}\sqcup\{y_j\}]$, by writing $\alpha = \begin{pmatrix} x_i & y_j \\ b_i & d_j \end{pmatrix}$ we mean that $x_i\alpha=b_i$, $y_j\alpha=d_j$ for all $i\in I, j\in J$, and that this map is the endomorphism obtained by extension of the underlying map on the basis using the free basis property of our independence algebra. We also notice that the sets $\{b_i\}$ and $\{d_j\}$ are not necessarily independent and that their intersection is not necessarily empty, and if both are subsets of the subalgebra $\AlgB$, then $\alpha$ lies in $\TAB$;
    \item by writing $t(\ob{x_i})$, we mean that the $n$-ary term $t$ depends upon the set $\set{x_i}$, such that there exist a subset $\set{x_{i_1},\dots,x_{i_n}}\subseteq \set{x_i}$ with $t(\ob{x_i}) = t(x_{i_1},\dots,x_{i_n})$. This will be used by saying that for any $c\in C=\clotX[\set{x_i}]$ there exists a term $t$ such that $c=t(\ob{x_i})$;
    \item since $\TAB$ is not a monoid in general as noted in Lemma \ref{TAB almost never monoid}, we will denote by $\TAB^1$ the monoid obtained by adjoining an identity to $\TAB$.
\end{itemize}

We finish this section by giving a few lemmas that provide some useful initial properties of $\TAB$ inherited from those of $\EndA$.
The omitted proofs are well known and can be found in \cite{AF04,Gould95}.
\begin{lemma}
Let $\alpha\in \EndA$ and $\{x_i\}=X\subseteq A$. If $\{x_i\alpha\}$ is independent, then $X$ is an independent set.
\end{lemma}
Following this, we define a \emph{preimage basis} of $\alpha\in\EndA$ to be an independent set $X\subseteq A$ such that $X\alpha$ is a basis for $\im\alpha$.
Notice that, by writing $\im\alpha=\clotX[\set{x_i\alpha}]$, if $\AlgA$ has constants then we have that $\rho(\alpha)=0$ if and only if $I=\emptyset$, in which case the empty set is a preimage basis for $\alpha$. 

\begin{lemma}\label{subalgebras of same rank are isomorphic}
In an independence algebra, all subalgebras of a given rank are isomorphic.
\end{lemma}

\begin{lemma}\label{isom of images gives RL related maps}
Let $\alpha,\beta\in \TAB$ be such that $\im\alpha\cong\im\beta$. Then there exist $\gamma,\mu\in\TAB$ such that: 
\begin{itemize}
    \item $\im\gamma=\im\beta$ and $\ker\gamma=\ker\alpha$; and
    \item $\im\mu=\im\alpha$ and $\ker\mu=\ker\beta$.
\end{itemize}
\end{lemma}
\begin{proof}
Since $\im\alpha\cong\im\beta$, there exists an isomorphism $\phi\colon A\alpha\to A\beta$.
Define $\gamma, \mu\in\TAB$ by $\gamma=\alpha\phi$ and $\mu=\beta\phi^{-1}$.
Then, $\im\gamma = A\gamma = (A\alpha)\phi = A\beta = \im\beta$ and similarly $\im\mu = \im\alpha$.
Additionally, we have that $\ker\gamma\supseteq\ker\alpha$ by definition of $\gamma$ and for the reverse inclusion, we have that for any $a,b\in A$, $(a,b)\in\ker(\gamma)$ if and only if $(a\alpha,b\alpha)\in\ker\phi$. 
However, since $\phi$ is injective, this is equivalent to $a\alpha=b\alpha$, that is, $(a,b)\in \ker\alpha$ and thus $\ker\gamma\subseteq\ker\alpha$.
Therefore $\ker\gamma=\ker\alpha$ and using similar arguments we have $\ker\mu=\ker\beta$ which concludes the proof.
\qed\end{proof}

\begin{lemma}\label{rank of product of maps less than min of ranks}
Let $\alpha,\beta\in\TAB$. Then $\rho(\alpha\beta)\leq \min\set{\rho(\alpha), \rho(\beta)}$.
\end{lemma}
\begin{proof}
From $A\alpha\beta \subseteq B\beta \subseteq A\beta$, we get that $\rho(\alpha\beta)\leq \rho(\beta)$. Similarly, $\dim(A\alpha\beta) \leq \dim((A\alpha)\beta)\leq \dim(A\alpha)$ and thus $\rho(\alpha\beta)\leq\rho(\alpha)$ and the claim follows from the two inequalities.
\qed\end{proof}

\section{Regular elements of $\TAB$}\label{section on Q}
It is well-known that for any independence algebra $\AlgA$, its endomorphism monoid $\EndA$ is regular \cite[Prop. 4.7]{Gould95}.
Thus in the case where $\AlgB=\AlgA$, it follows that $\TAB$ is regular. 
On the other hand, it is easy to see that if we take $\AlgB$ to be the constant subalgebra $\Csts$ whenever this is non-empty, we have that $\TAB$ is a left-zero semigroup. 
Indeed, since any homomorphism has to be act as the identity on $\Csts$, then for any $\alpha,\beta\in\TAB$ we have that $\im\alpha\subseteq \Csts$ and thus $\alpha\beta = \alpha$.
Besides, when $\Csts=\emptyset$ and $\dim\AlgB=1$ we have that $\TAB$ is isomorphic to $G\wr_{_{n}}\!\set{c_1}$, where $G$ is the group of unary term operations which do not have constant image, $c_1\colon\set{1,\dots,n}\to \set{1,\dots,n}$ is the constant map with value $1$ and the multiplication is given by
\begin{align*}\left((g_1,\dots,g_n),c_1\right)\left((h_1,\dots,h_n),c_1\right) &= \left((g_1h_{1 c_1},\dots,g_nh_{n c_1}),c_1\right)\\ &= \left((g_1h_1,\dots,g_nh_1),c_1\right).\end{align*}
In both these cases, $\TAB$ is a regular semigroup but we will show that $\TAB$ is not regular when $\AlgB$ does not satisfy the conditions given above.

To see that, first notice that if $\alpha\in\TAB$ is any regular element, then there exists $\gamma\in\TAB$ such that $\alpha=\alpha\gamma\alpha$ and from this, we obtain that 
$$ A\alpha=(A\alpha\gamma)\alpha \subseteq B\alpha.$$
We therefore define a set $Q$ which will contain all regular elements as follows:
$$ Q = \set{\alpha\in\TAB\mid A\alpha\subseteq B\alpha}.$$
It is clear that the condition $A\alpha\subseteq B\alpha$ on the elements of $Q$ can be rewritten as $A\alpha = B\alpha$ or $(A\priv B)\alpha\subseteq B\alpha$, and these equivalent conditions will be equally used throughout the paper to define an element of $Q$.

\begin{lemma}
The set $Q$ is a right ideal and the maximal regular subsemigroup of $\TAB$ with respect to containment.
\end{lemma}
\begin{proof}
As per the remark above, any regular element of $\TAB$ lies in $Q$. 
Moreover, for any $\alpha\in Q, \beta\in \TAB$ we have that 
$$A(\alpha\beta) = (A\alpha)\beta\subseteq (B\alpha)\beta = B\alpha\beta,$$
and thus $Q$ is a right ideal.

Let us now show that any $\alpha\in Q$ is a regular element. 
Let $B\alpha = \clotX[\set{b_i\alpha}]$, where, according to our convention, $\set{b_i}\subseteq B$ is a preimage basis, and take an element $c\in\clotX[\set{b_i}]$, which exists as $\AlgB$ is non-empty.
Since $A\alpha\subseteq B\alpha$, then for any $a\in A$ there exists a term $t$ such that $a\alpha = t(\ob{b_i\alpha})=t(\ob{b_i})\alpha$. 
Now if $A=\clotX[\set{b_i}\sqcup\set{x_j}]$ and we define $u_j = x_j\alpha = t_j(\ob{b_i\alpha})$, then we have that $$\alpha = \pmap{ b_i & x_j \\ b_i\alpha & u_j}.$$ 
Additionally, let $\set{a_j}$ be such that $A=\clotX[\set{b_i\alpha}\sqcup\set{a_j}]$, and define $\gamma\in \TAB$ by:
$$\gamma = \pmap{ b_i\alpha & a_j \\ b_i & c }.$$
Then we have that $\gamma\in Q$ since $A\gamma = \clotX[b_i] = B\gamma$.
Finally, we can see that $\alpha=\alpha\gamma\alpha$. 
Indeed, this is clear for any $b\in \clotX[b_i]$, but also for $a\in A$ we have 
$$a\alpha\gamma\alpha = t(\ob{b_i\alpha})\gamma\alpha= t(\ob{(b_i\alpha)\gamma})\alpha = t(\ob{b_i})\alpha = a\alpha,$$ 
which finishes the proof that any element in $Q$ is regular.
\qed\end{proof}

For ease of use, the set of non-regular maps, that is, the set $\TAB\priv Q$ will be denoted by $\Qc$.

\begin{corollary}\label{cor tab not reg in general}
For $\AlgB\neq\AlgA$, the semigroup $\TAB$ is not regular if and only if $\dim B\geq 2$ or $\Csts\nonempty$ and $\dim B=1$.
\end{corollary}
\begin{proof}
If $\dim B\geq 2$ or $\dim B=1$ and $\Csts\nonempty$, there exists at least two elements $b_1,b_2\in B$ such that $b_1\notin \clotX[b_2]$ since we can either take two independent elements, or take $b_1$ to be an independent element and $b_2$ a constant.
Now let $B=\clotX[\set{x_j}]$, $A=\clotX[\set{y_i}\sqcup\set{x_j}]$ and define $\alpha\in\TAB$ by
$$ \alpha=\pmap{y_i & x_j \\ b_1 & b_2}.$$
Then we have that $B\alpha=\clotX[b_2]\subsetneq \clotX[b_1,b_2] = A\alpha$ and thus $\alpha$ is not in $Q$.

Additionally, the comments at the beginning of Section \ref{section on Q} give us the contrapositive statement, hence the equivalence.
\qed\end{proof}

Following Corollary \ref{cor tab not reg in general}, we will restrict ourselves to the cases where $\TAB$ is not regular, since if $\TAB$ is regular, then we automatically obtain Green's relations and the ideal structure from that of $\EndA$.
Therefore, from now on we assume that $\AlgB\neq\AlgA$ and that $\dim B\geq 2$ or $\Csts\neq\emptyset$ and $\dim B=1$.
As a consequence, a similar argument to the one mentioned in the proof will be often used, namely, the fact that there exist two distinct elements $b_1$ and $b_2$ in $B$ such that the former does not lie in the subalgebra generated by the latter.
This will allow us to treat in a similar way algebras of a given rank that contain constants with algebras of a rank one higher that do not have them, as long as we define our maps to be the identity on the element $b_2$ whenever this is relevant.

\begin{remark} \label{remark structure Q}
A few notes on the structure of $Q$ worth mentioning are the following:
\begin{itemize}
    \item \emph{$Q$ is always nonempty:} let $B=\clotX[\set{b_i}]$ and $A=\clotX[\set{b_i}\sqcup\set{a_j}]$, then the map $\alpha=\pmap{b_i & a_j \\ b_i & b_1}$ is an idempotent and is clearly in $Q$. 
    \item \emph{$Q$ is not a left ideal:} let $b_1, b_2\in B$ such that $b_1\notin\clotX[b_2]$ and define the algebras $C=\clotX[b_1,b_2]$, $B=C\sqcup\clotX[\set{y_j}]$, and $A=B\sqcup\clotX[\set{x_i}]$ together with the following maps:
$$ \alpha = \pmap{x_i & y_j & b_1 & b_2 \\ b_1 & b_2 & b_2 & b_2} \quad \text{and} \quad \beta = \pmap{x_i & y_j & b_1 & b_2 \\ b_1 & b_1 & b_1 & b_2}. $$
Then it follows that $\alpha\notin Q$ and $\beta\in Q$ are such that $\alpha\beta\notin Q$.
    \item \emph{Any map in $Q$ has a preimage basis in $B$:} for $\alpha\in Q$ we have that $\im\alpha=A\alpha=B\alpha$, so there exist $\set{b_i}\subseteq B$ such that $\im\alpha=\clotX[\set{b_i\alpha}]$, and then $\set{b_i}$ is a preimage basis for $\alpha$. 
    From now on, this fact will be used without explicit mention.
\end{itemize}
\end{remark}

Since regular maps will be of utmost importance in the description of Green's relations in $\TAB$, some lemmas are given here for later use.
The first one gives us a sufficient and necessary condition on when the product of two maps lies in $Q$.

\begin{lemma}\label{condition product to be regular}
For $\alpha, \beta\in \TAB$, $\alpha\beta\in Q$ if and only if for all $x\in A$, there exists a $y\in B$ such that $(x\alpha,y\alpha)\in \ker\beta\cap (B\times B)$.
\end{lemma}
\begin{proof}
From the remarks above, if $\alpha\beta\in Q$, then there exists a preimage basis of $\alpha\beta$ in $B$, say $\set{b_i}$. 
Therefore, for any $a\in A$, we have that $a\alpha\beta \in \clotX[\set{b_i\alpha\beta}]$ and thus there exists a term $t$ such that $a\alpha\beta = t(\ob{b_i\alpha\beta}) = t(\ob{b_i})\alpha\beta$.
Since $t(\ob{b_i})\in B$ and $\left(a\alpha, t(\ob{b_i})\alpha\right)\in \ker\beta\cap (B\times B)$, the first direction is proved.

For the converse, we have that $A\alpha\beta = \set{x\alpha \mid x \in A}\beta$. 
However, each element of the set $\set{x\alpha \mid x\in A}$ lies in $B$ and is $\ker\beta$-related to some element $y\alpha$ in $B\alpha$ by assumption.
Thus $A\alpha\beta = \set{y\alpha\mid y\in B}\beta = B\alpha\beta$ which gives us $\alpha\beta\in Q$.
\qed\end{proof}

The next lemma shows that we can always create a map in $Q$ from a non-regular one with the same image, and that the converse is also true given that the rank is not minimal.
For this, let us denote by $e$ the minimal rank of a subalgebra of $\AlgA$, that is, $e=0$ if $\AlgA$ has constants, and $e=1$ otherwise. 
As a direct consequence, $e$ is also the minimal rank of any map in $\TAB$.

\begin{lemma}\label{Reg and non-reg maps with same image}
For any map $\alpha\in \Qc$, there exists a map $\alpha'\in Q$ such that $\im\alpha'=\im\alpha$.
Similarly, for any map $\beta\in Q$ such that $\rho(\beta)>e$, there exists $\beta'\in \Qc$ such that $\im\beta'=\im\beta$.
\end{lemma}
\begin{proof}
Let us assume first that $\alpha\in \Qc$.
Then $B\alpha=\clotX[\set{b_i\alpha}]$ (with $I$ possibly empty) and we extend this to a basis $\set{b_i\alpha}\sqcup\set{a_j\alpha}$ of $\im\alpha$. 
Since $\dim B\geq \Card{I} + \Card{J}$, we can write $B=\clotX[\set{c_i}\sqcup\set{x_j}\sqcup\set{y_k}]$ where the set $K$ is possibly empty.
Finally letting $A = \clotX[\set{z_s}\sqcup\set{c_i}\sqcup\set{x_j}\sqcup\set{y_k}]$ and defining $\alpha'$ as the following:
$$ \alpha' = \pmap{z_s & c_i & x_j & y_k \\ z_s\alpha & b_i\alpha & a_j\alpha & y_k\alpha},$$
we have that 
$$A\alpha'=\im\alpha=\clotX[\set{b_i\alpha}\sqcup\set{a_j\alpha}] = \clotX[\set{c_i}\sqcup\set{x_j}]\alpha'\subseteq B\alpha',$$
and thus $\alpha'\in Q$ with the same image as $\alpha$.

Now assume that $\beta\in Q$ is such that $\rho(\beta)>e$. 
Then, there exist two elements $b_1, b_2\in B$ such that $b_1\beta\notin \clotX[b_2\beta]$ together with a set $\set{x_j}\subseteq B$ such that $A\beta = B\beta = \clotX[b_1\beta, b_2\beta]\sqcup\clotX[\set{x_j\beta}]$, and we write $B=\clotX[b_1,b_2]\sqcup\clotX[\set{x_j}]$, $A=\clotX[\set{a_i}]\sqcup B$ and define 
$$\beta'=\pmap{a_i & \set{b_1, b_2} & x_j \\ b_1\beta & b_2\beta & x_j\beta}.$$
Then $A\beta'=\clotX[b_1\beta,b_2\beta]\sqcup\clotX[\set{x_j\beta}]$ and $B\beta'=\clotX[b_2\beta]\sqcup\clotX[\set{x_j\beta}]$.
By the exchange property, we have that $b_1\beta \notin B\beta'$ which gives us that $A\beta'\neq B\beta'$, that is $\beta'\notin Q$, and also $\im\beta'=\im\beta$ as required.
\qed\end{proof}

When we consider a non-regular map $\alpha$, even though we do not have that $A\alpha = B\alpha$, nevertheless we can have that $A\alpha\cong B\alpha$, which happens precisely when $A\alpha$ and $B\alpha$ have the same dimension. 
As an example, consider $A = \clotX[\set{x_i}]$, $B=\clotX[\set{x_i}\priv\set{x_1}]$ where $I=\N$, and define $\alpha\in\TAB$ by $x_k\alpha = x_{k+1}$. 
Then clearly $x_2\in A\alpha\priv B\alpha$, so that $\alpha\notin Q$ but we still have $\rho(\alpha)=\dim B=\dim(B\alpha)$.
This peculiar behaviour of non-regular maps can however only happen if $\alpha$ has infinite rank. This is made clear in the following lemma.

\begin{lemma}\label{lem rank of restriction wrt finiteness}
Let $\alpha\in\TAB$ be such that $A\alpha\cong B\alpha$. Then either $\alpha\in Q$, or $\rho(\alpha)\geq \aleph_0$.
\end{lemma}

\begin{proof}
Assume that $A\alpha\cong B\alpha$. 
If $A\alpha = B\alpha$ then $\alpha\in Q$.
Otherwise, we have that  $\rho(\alpha\lvert_B) = \rho(\alpha)$ by Lemma \ref{subalgebras of same rank are isomorphic} and we let $B\alpha = \clotX[\set{b_k\alpha}]$.
Since $B\alpha\subsetneq A\alpha$, there exist elements $\set{a_j\alpha}\in\im\alpha\priv B\alpha$ and we can write $A\alpha = \clotX[\set{b_k\alpha}\sqcup\set{a_j\alpha}]$ with $\Card{J}\geq 1$. 
But then 
$$\Card{K} = \rho(\alpha\lvert_B) = \rho(\alpha) = \Card{K} + \Card{J},$$
which can only happen if $\Card{K}=\rho(\alpha\lvert_B)=\rho(\alpha)$ is infinite.
\qed\end{proof}

As an immediate consequence, the contrapositive statement tells us that a non-regular map $\alpha$ with finite rank necessarily satisfy that $\rho(\alpha\lvert_B) = \dim(B\alpha) < \dim(A\alpha) = \rho(\alpha)$.

The last lemma for this section concerns the equivalence between a map being regular and the existence of a map acting as a left identity on it.

\begin{lemma}\label{Regular maps factorised by idp}
Let $\alpha\in \TAB$. Then the following are equivalent:
\begin{enumerate}
    \item $\alpha$ is regular;
    \item $\alpha=\eta\alpha$ for some $\eta$ idempotent in $\TAB$ with $\rho(\eta) = \rho(\alpha)$; and
    \item $\alpha=\gamma\alpha$ for some $\gamma\in\TAB$.
\end{enumerate}
\end{lemma}
\begin{proof}
\textit{$1\Rightarrow 2$:} Since $\alpha$ is regular, then $\alpha=\alpha\beta\alpha$ for some $\beta\in\TAB$ and taking $\eta=\alpha\beta$ gives the desired result. \\
\textit{$2\Rightarrow 3$:} Putting $\gamma=\eta$ gives the result immediately. \\
\textit{$3\Rightarrow 1$:} If $\alpha = \gamma\alpha$, then $A\alpha = (A\gamma)\alpha \subseteq B\alpha$ and thus $\alpha\in Q$.
\qed\end{proof}

\section{Green's relations}\label{section on basic Greens}

In order to better understand the structure of $\TAB$, as for any semigroup, we start by looking at Green's relations.
There are three relevant semigroups in question here: $\EndA$, $\TAB$ and $Q$.
To avoid confusion, where the relation is on $\EndA$ or $Q$, we use a subscript (we consider $\TAB$ as our base case and so do not use a subscript here).
For example, $\GRE$, $\GR$ and $\GRQ$ denote Green's relation $\GR$ on $\EndA$, $\TAB$ and $Q$ respectively.
On the monoid $\EndA$, the Green's relations have been studied extensively (see for example \cite{Gould95}) where we have that $\alpha\GRE\beta$ if and only if $\ker\alpha = \ker\beta$, $\alpha\GLE\beta$ if and only if $\im\alpha=\im\beta$, and $\alpha\GDE\beta$ if and only if $\rho(\alpha)=\rho(\beta)$.
It is well known that a regular subsemigroup inherits Green's relations $\GR$ and $\GL$ from the larger semigroup and since $Q$ is a regular subsemigroup of $\EndA$, we get that $\GRQ=\GRE$, $\GLQ=\GLE$ and thus also $\GDQ = \GDE$.
The description of the Green's relations in the semigroup $\TAB$ is however slightly different.
This section will be devoted to proving adapted versions of the statements Sullivan, Sanwong and Sommanee gave in \cite{SS08,Sul07} in a more general approach that covers all types of independence algebras, using similar ideas to those they developed.

We start by showing that $\GR$ behaves the same way in all of $\TAB$ as given by:

\begin{proposition}\label{GR in TAB}
If $\alpha, \beta\in \TAB$, then $\alpha = \beta\mu$ for some $\mu\in \TAB$ if and only if $\ker\beta \subseteq \ker\alpha$. Consequently, $\alpha\GR\beta$ in \TAB if and only if $\ker\alpha=\ker\beta$.
\end{proposition}
\begin{proof}
Clearly if $\alpha = \beta\mu$ then $\ker\beta\subseteq \ker\alpha$. 
Now suppose that $\ker\beta\subseteq \ker\alpha$, and let $A\beta = \clotX[\set{a_i\beta}]$. 
Now let $A = \clotX[\set{a_i\beta}\sqcup\set{x_j}]$ and define $\gamma\in\TAB$ by:
$$ \gamma = \pmap{ a_i\beta & x_j \\ a_i\alpha & x_j\beta}.$$
By definition of $\set{a_i\beta}$, for any $z\in A$, there exists a term $t$ such that $z\beta = t(\ob{a_i\beta}) = t(\ob{a_i})\beta$ which shows that $(z,t(\ob{a_i}))\in \ker\beta \subseteq \ker\alpha$, and thus
$$z\beta\gamma = t\left(\ob{a_i\beta}\right)\gamma = t\left(\ob{(a_i\beta)\gamma}\right) = t\left(\ob{a_i\alpha}\right) = t(\ob{a_i})\alpha = z\alpha.$$
Therefore we have that $\alpha = \beta\gamma$.

It follows from this that if $\ker\alpha = \ker\beta$ then $\alpha\GR\beta$. 
Conversely if $\alpha\GR\beta$ then $\alpha\GRE\beta$ and so $\ker\alpha = \ker\beta$ using the remarks made previously.
\qed\end{proof}

Notice that regular maps can only be related to other regular maps since $\GR\subseteq \GD$ and in any semigroup either all elements in a $\GD$-class are regular, or none of them is.

The relation $\GL$, however, does not behave exactly as in $\EndA$ and is more restrictive on the non-regular part of $\TAB$.
\begin{proposition}\label{GL in TAB}
If $\alpha\in \TAB$ and $\beta\in Q$, then $\alpha= \lambda\beta$ for some $\lambda\in \TAB$ if and only if $\im\alpha \subseteq \im\beta$. 
Consequently, $\alpha\GL\beta$ in \TAB if and only if $\alpha=\beta$ or (~$\im\alpha = \im\beta$ and $\alpha, \beta\in Q$).
\end{proposition}
\begin{proof}
Let $\alpha\in\TAB$ and $\beta\in Q$.
Clearly if $\alpha = \lambda\beta$, then $\im\alpha \subseteq\im\beta$. 
Conversely, let us assume that $\im\alpha\subseteq \im\beta$ and write $A\alpha = \clotX[\set{a_i\alpha}]$ with $\set{a_i}\subseteq A$. 
Since $\beta\in Q$, we have that $\im\alpha\subseteq B\beta$ and thus there exist $\set{f_i}\subseteq B$ such that $a_i\alpha = f_i\beta$ for all $i\in I$, hence the set $\{f_i\beta\}$ is independent.
Now we take $\{f_j\beta\}$ such that $A\beta = \clotX[\{f_i\beta\}\sqcup\{f_j\beta\}]$ and let $\{x_k\}$ and $\{y_\ell\}$ subsets of $A$ be such that $A=\clotX[\{a_i\}\sqcup\{x_k\}]$ and 
$A = \clotX[\{f_i\}\sqcup\{f_j\}\sqcup\{y_\ell\}]$. 
From this we can write the maps $\alpha$ and $\beta$ as:
$$
\alpha = \begin{pmatrix}a_i & x_k \\ a_i\alpha & t_k(\ob{a_i\alpha}) \end{pmatrix} \qquad \text{and} \qquad 
\beta = \pmap{f_i & f_j & y_\ell \\ f_i\beta & f_j\beta & y_\ell\beta}.
$$
for some terms $t_k$.
By defining $\lambda\in \TAB$ by 
$$ \lambda = \pmap{a_i & x_k \\ f_i & t_k(\ob{f_i})},$$
we get that $a_i\lambda\beta = f_i\beta = a_i\alpha$ and $x_k\lambda\beta = t_k(\ob{f_i})\beta = t_k(\ob{f_i\beta}) = t_k(\ob{a_i\alpha}) = x_k\alpha$.
Thus $\alpha = \lambda\beta$, which concludes the first statement of the proposition.

Now, let us assume that $\alpha\GL\beta$ in $\TAB$. 
Then, there exist $\lambda, \lambda'\in \TAB^1$ such that $\alpha = \lambda\beta$ and $\beta = \lambda'\alpha$. 
If $\lambda=1$ or $\lambda'=1$ then we have that $\alpha = \beta$.
Otherwise, we have that 
$$\alpha = \lambda\lambda' \alpha\quad \text{and} \quad \beta = \lambda'\lambda\beta,$$
from which we get that $\alpha,\beta\in Q$ using Lemma \ref{Regular maps factorised by idp}.
From the first part of the proposition we also have that $\im\alpha\subseteq \im\beta$ and $\im\beta\subseteq \im\alpha$, hence the equality.
The converse follows directly from the description of $\GLQ$ given above whenever $\alpha\neq \beta$.
\qed\end{proof}

From $\GR$ and $\GL$ we can deduce the description of the relation $\GH$ which is the equality relation on $\Qc$ and consists of the set of maps with the same kernel and image on $Q$.
The relations $\GD$ and $\GJ$ are given by the following two propositions.

\begin{proposition}\label{GD in TAB}
If $\alpha,\beta \in \TAB$, then $\alpha\GD\beta$ in $\TAB$ if and only if $\ker\alpha = \ker\beta$ or ($\rho(\alpha)=\rho(\beta)$ and $\alpha, \beta\in Q$).
\end{proposition}
\begin{proof}
Assume that $\alpha\GD\beta$. 
Then there exists $\gamma\in \TAB$ such that $\alpha\GR\gamma\GL\beta$.
By Propositions \ref{GR in TAB} and \ref{GL in TAB} we have that $\ker\alpha = \ker\gamma$, and either $\gamma =\beta$ or $\gamma, \beta\in Q$ have the same image.
In the case where $\gamma = \beta$ we obtain that $\ker\alpha = \ker\gamma = \ker\beta$, as expected. 
So let us now assume that $\im\gamma = \im\beta$ and $\gamma, \beta\in Q$. 
Then $\alpha$ is also in $Q$ since $\alpha\GR\gamma$ and we also have that
$$ \rho(\beta) = \rho(\gamma) = \dim(A/\ker\gamma) = \dim(A/\ker\alpha) = \rho(\alpha),$$
as required.

Conversely, if $\ker\alpha = \ker\beta$, then $\alpha\GR\beta$ and thus $\alpha\GD\beta$.
Otherwise, assume $\rho(\alpha) = \rho(\beta)$ and $\alpha, \beta\in Q$. 
Therefore, by Lemmas \ref{subalgebras of same rank are isomorphic} and \ref{isom of images gives RL related maps} we have that $\im\alpha \cong \im\beta$ and then there exists $\gamma\in\TAB$ such that $\im\gamma=\im\beta$, and $\ker\gamma=\ker\alpha$. 
Furthermore, since $\alpha\in Q$, we get that $\gamma\in Q$ by the note following Proposition \ref{GR in TAB}, which finishes the proof that $\alpha\GR\gamma\GL\beta$ and therefore $\alpha\GD\beta$.
\qed\end{proof}

\begin{proposition}\label{GJ in TAB}
If $\alpha,\beta\in \TAB$, then $\alpha= \lambda\beta\mu$ for some $\lambda\in \TAB$ and $\mu\in\TAB^1$ if and only if $\rho(\alpha) \leq \dim(B\beta)$. Consequently, $\alpha\GJ\beta$ in $\TAB$ if and only if one of the following happens:
\begin{itemize}
    \item $\ker\alpha = \ker\beta$, or
    \item $\rho(\alpha) = \dim(B\alpha) = \dim(B\beta) = \rho(\beta)$.
\end{itemize}
\end{proposition}

\begin{proof}
Let us assume first that $\alpha = \lambda\beta\mu$ for some $\lambda\in\TAB$ and $\mu\in\TAB^1$.
Then $A\alpha = (A\lambda)\beta\mu \subseteq (B\beta)\mu$ which gives that $\rho(\alpha) = \dim(A\alpha)\leq \dim\left((B\beta)\mu\right) \leq \dim(B\beta)$.
Conversely, let us assume that $\rho(\alpha) \leq \dim(B\beta)$.
If $\rho(\alpha) = 0$, then clearly $\alpha = \alpha\beta$ and we obtain the result.
Assuming now that $\rho(\alpha)\geq 1$ and writing $A\alpha = \clotX[\set{x_i\alpha}]$, we have by the first hypothesis on the rank of $\alpha$ that there exists an independent set $\set{b_i}\subseteq B$ such that $A\beta = \clotX[\{b_i\beta\}\sqcup \{y_k\beta\}]$. 
Thus we can write $A = \clotX[\{x_i\}\sqcup\{x'_j\}] = \clotX[\{b_i\}\sqcup\{y_k\}\sqcup\{y_\ell\}]$ and we get that 
$$ \alpha = \pmap{x_i & x'_j\\ x_i\alpha & u_j(\ob{x_i\alpha})} \quad \text{and} \quad 
\beta = \pmap{b_i & y_k & y_\ell \\ b_i\beta & y_k\beta & v_\ell(\ob{b_i\beta}, \ob{y_k\beta})},$$
for some terms $u_j$ and $v_\ell$. Now if we also extend $\{b_i\beta\}$ into a basis of $A$ via $\{z_m\}$ and define $\lambda$ and $\mu$ in $\TAB$ by the following:
$$ \lambda = \pmap{x_i & x'_j\\ b_i & u_j(\ob{b_i})} \quad \text{and} \quad 
\mu = \pmap{b_i\beta & z_m \\ x_i\alpha & x_1\alpha},$$
then it can be easily seen that $\alpha=\lambda\beta\mu$.

In order to prove the second part of the proposition, let us first assume that $\alpha\GJ\beta$, that is, there exist $\lambda,\mu,\lambda',\mu' \in \TAB^1$ such that $\alpha = \lambda\beta\mu$ and $\beta=\lambda'\alpha\mu'$.
If $\lambda=\lambda'=1$, then we have that $\alpha\GR\beta$ and thus Proposition \ref{GR in TAB} gives us that $\ker\alpha=\ker\beta$.
If we have that only one of $\lambda$ and $\lambda'$ is $1$, then we can find $\gamma, \gamma'\in\TAB$ and $\delta,\delta'\in \TAB^1$ such that $\alpha = \gamma\beta\delta$ and $\beta=\gamma'\alpha\delta'$. 
Indeed, assuming without loss of generality that $\lambda=1$, we have $\alpha = \beta\mu = \lambda'\alpha\mu'\mu = \lambda'\beta(\mu\mu'\mu)$ and we thus can take $\gamma=\gamma'=\lambda'$, $\delta=\mu\mu'\mu$ and $\delta'=\mu'$.
By the previous part of the proposition we therefore have that $\rho(\alpha)\leq\dim(B\beta)$ and $\rho(\beta)\leq \dim(B\alpha)$.
Therefore we get that:
$$ \dim (A\alpha) = \rho(\alpha) \leq \dim (B\beta) \leq \dim (A\beta) = \rho(\beta)\leq\dim(B\alpha) \leq \dim(A\alpha),$$
which forces the equalities.

Conversely, assume that $\ker\alpha=\ker\beta$, or $\rho(\alpha) = \dim(B\beta) = \dim(B\alpha) = \rho(\beta)$.
If the first condition holds, then $\alpha\GR\beta$ by Proposition \ref{GR in TAB} and thus $\alpha\GJ\beta$ since $\GR\subseteq\GJ$. 
On the other hand, if the second condition holds, then we use the first part of the proof in both directions in order to obtain the desired result.
\qed\end{proof}

\begin{remark}
In the proof of Proposition \ref{GJ in TAB} given above, notice that the maps $\lambda$ and $\mu$ are constructed in such a way that $\mu\in Q$, and whenever $\set{x_i}\subseteq B$ we also have $\lambda\in Q$.
\end{remark}

From Lemma \ref{lem rank of restriction wrt finiteness}, we know that the condition $\rho(\alpha) = \dim(B\alpha)$ is equivalent to $\alpha\in Q$ only if $\alpha$ has finite rank, which means that $\GD = \GJ$ whenever $\AlgB$ is finite dimensional. 
On the other hand, if $\AlgB$ is infinite dimensional, then there exist non-regular maps of infinite rank that are $\GJ$-related to regular maps of the same rank, which shows that in that case $\GD\subsetneq \GJ$.
However, on the regular subsemigroup $Q$ these relations always coincide and the description also corresponds to that of $\GDE$ relating maps of the same rank, as is given by the following: 

\begin{lemma}\label{GJ=GD in Q}
Let $\alpha, \beta\in Q$. Then $\alpha = \lambda\beta\mu$ for some $\lambda,\mu\in Q$ if and only if $\rho(\alpha) \leq \rho(\beta)$.
Consequently, $\alpha\GJQ\beta$ if and only if $\rho(\alpha) = \rho(\beta)$, and thus $\GJQ = \GDQ$.
\end{lemma}
\begin{proof}
Clearly if $\alpha=\lambda\beta\mu$ then $\rho(\alpha)\leq \rho(\beta\mu) \leq \rho(\beta)$ by Lemma \ref{rank of product of maps less than min of ranks}. 
Conversely, assume that $\rho(\alpha)\leq\rho(\beta)$.
Since both $\alpha$ and $\beta$ lie in $Q$, we have that $A\alpha=B\alpha$ and $\rho(\beta) =\dim(A\beta) = \dim(B\beta)$. 
Thus, using the same maps $\lambda$ and $\mu$ constructed in the proof of Proposition \ref{GJ in TAB} together with the remark above, we get that $\alpha = \lambda\beta\mu$ with $\lambda,\mu\in Q$.

The second part of the lemma now follow directly from the definition of maps being $\GJ$-related in $Q$ together with the characterisation of $\GD$-classes in $Q$ given in Proposition \ref{GD in TAB}.
\qed\end{proof}

\section{Ideals}\label{section on ideals}

In the same way that Section \ref{section on basic Greens} was simply generalising the approach to Green's relations exhibited in the cases of vector spaces and sets, this section is generalising the description of the ideals of \TAB using the same ideas developed by Sullivan and Mendes-Gonçalves \cite{MGS10,Sul07}.
It is well-known \cite{Gould95} that the ideals of the endomorphism monoid of an independence algebra are precisely the sets $\set{\alpha\in \EndA \mid \rho(\alpha)< k}$ for each $k\leq (\dim A)^+$, where we denote by $\kappa^+$ the successor cardinal of $\kappa$.
However, in the context of \TAB, the ideals are not solely determined by ranks. 
Nevertheless, in the subsemigroup $Q$, all ideals are in one-to-one correspondence with the cardinals smaller than $(\dim B)^+$.

Recall that $e$ denotes the smallest rank of a non-empty subalgebra of $\AlgA$ or equivalently the smallest rank of a map in \TAB.

\begin{proposition}[see {\cite[Theorem 8]{Sul07}}]
The ideals of $Q$ are precisely the sets 
$$ Q_r = \{ \alpha\in Q\mid \rho(\alpha)< r \} $$
where $e < r\leq (\dim B)^+$.
\end{proposition}
\begin{proof}
Let $e<r\leq (\dim B)^+$. Let $\alpha\in Q_r$ and $\beta\in Q$. 
Then from Lemma \ref{rank of product of maps less than min of ranks} we have that $\rho(\beta\alpha), \rho(\alpha\beta)\leq \rho(\alpha)$ and then both $\alpha\beta$ and $\beta\alpha$ lie in $Q_r$, which shows that $Q_r$ is an ideal of $Q$.

Conversely, let us assume that $I$ is an ideal of $Q$ and let $r$ be the least cardinal strictly greater than the rank of all maps in $I$, so that $e<r\leq (\dim B)^+$. 
We show that $I =Q_r$. Clearly $I\subseteq Q_r$. 
For the reverse inclusion, let $\beta\in Q_r$.
If $\rho(\alpha) < \rho(\beta)$ for all $\alpha\in I$, then this means that $\rho(\beta) \geq r$ by minimality of $r$, which contradicts our assumption that $\beta\in Q_r$. 
Thus, there exists some $\alpha\in I$ such that $\rho(\beta)\leq \rho(\alpha)$.
By Lemma \ref{GJ=GD in Q}, it follows that there exist $\lambda, \mu\in Q$ such that $\beta= \lambda\alpha\mu$.
This gives us that $\beta\in I$ since $I$ is an ideal of $Q$ and thus $Q_r\subseteq I$. 
Therefore $I=Q_r$, completing the proof.
\qed\end{proof}

Following the usual definition in \EndA, we define the sets $T_k$ for any $ k>e$ by 
$T_k = \set{\alpha\in\TAB\mid\rho(\alpha)< k}$, which are easily seen to be ideals of $\TAB$ using Lemma \ref{rank of product of maps less than min of ranks}.
It is obvious that for all $k\geq (\dim B)^+$ we have that $T_k = T_{(\dim B)^+} = \TAB$.
Following \cite{MGS10,Sul07} we define for any non-empty subset $S$ of $\TAB$ the cardinal $r(S)$ and the subset $K(S)\subseteq \TAB$ as follows:
$$\begin{gathered}
    r(S) = \min \{\kappa \leq (\dim B)^+ \mid \dim(B\alpha) < \kappa, \ \forall \alpha\in S\}, \\
\text{and }    K(S) = \{ \beta \in \TAB \mid \ker \alpha \subseteq \ker\beta, \text{ for some }\alpha\in S\}.
\end{gathered}$$
Using Proposition \ref{GR in TAB}, we can also express $K(S)$ as 
$$K(S)=\set{\beta\in\TAB\mid \beta = \alpha\mu, \text{ for some }\alpha\in S\text{ and }\mu\in\TAB}.$$
From this, it is clear that if $\beta\in K(S)$ and $\lambda\in\TAB$, then we have that $\beta\lambda\in K(S)$ and thus $K(S)$ is a right ideal.

We now want to show that any ideal of $\TAB$ is of the form $T_{r(S)}\cup K(S)$ or $T_{r(S)^+}\cup K(S)$ for some non-empty subset $S\subseteq \TAB$.
The ideals $T_k$ are easily seen to satisfy this rule as $T_k = T_{r(S)}\cup K(S)$ for $S=T_k$. 
Indeed, let $S=T_k$. Then, for any $s<k$, $T_k$ contains the idempotent $\eta_s$ with image a subalgebra $\AlgC$ of $\AlgB$ that has rank $s$, from which we get that $r(S) = k$.
Also, if $\beta\in K(S)$, then $\beta = \alpha\mu$ for some $\alpha\in T_k$ and $\mu\in\TAB$. 
Hence $\rho(\beta)=\rho(\alpha\mu)\leq\rho(\alpha)<k$, which shows that $\beta\in T_k$ and thus $K(S)\subseteq T_k$. 
Therefore $T_k = T_{r(S)}\cup K(S)$ as claimed.

Before we can show that all ideals are of the form $T_{r(S)}\cup K(S)$ or $T_{r(S)^+}\cup K(S)$ for some non-empty set $S$, we first need to show that such sets are indeed ideals.

\begin{lemma}\label{Trs union K(S) are ideals}
For each non-empty subset $S$ of \TAB, the sets $T_{r(S)} \cup K(S)$ and $T_{r(S)^+}\cup K(S)$ are ideals of \TAB.
\end{lemma}
\begin{proof}
Let $\emptyset \neq S\subseteq \TAB$, and consider $r(S)$ and $K(S)$ as in their definition.
As mentioned earlier, we first have that $K(S)$ is a right ideal.
Now let $\beta \in K(S)$ and $\lambda\in \TAB$. 
Then there exists $\alpha\in S$ and $\mu\in\TAB$ such that $\beta=\alpha\mu$.
Thus
$$ \rho(\lambda\beta) = \dim(A\lambda\beta) \leq \dim(B\beta) = \dim(B\alpha\mu) \leq \dim(B\alpha) < r(S),$$
and therefore $\lambda\beta\in T_{r(S)}$. 
The result now follows from the fact that $T_{r(S)} \subseteq T_{r(S)^+}$ and that $T_\kappa$ is a two-sided ideal of \TAB for all $\kappa>e$.
\qed\end{proof}

In order to show the reverse statement, we need a small lemma beforehand which will become handy in order to choose an adequate set $S$ for each ideal of $\TAB$.

\begin{lemma}\label{non-reg map with smaller rank}
If $\alpha\in Q$ and $e\leq s<\rho(\alpha)$, then there exists $\lambda\in \TAB$ such that $\lambda\alpha\notin Q$ and $\dim(B\lambda\alpha)=s$.
\end{lemma}
\begin{proof}
Since $\alpha\in Q$, then $A\alpha = B\alpha= \clotX[\set{b_i\alpha}]$ for some $\set{b_i}\subseteq B$, and for $A = \clotX[\set{b_i}\sqcup\set{x_j}]$ the map $\alpha$ can be written as
$$ \alpha = \pmap{b_i & x_j \\ b_i\alpha & u_j(\ob{b_i\alpha})},$$
for some terms $u_j$.
Now take $\set{b'_k}\sqcup \set{b_1} \subseteq \set{b_i} $ such that $\Card{K} = s$ (which is possible by the assumption on the value of $s$), and take some basis element $z\in A\setminus B$.
Let $A=\clotX[\set{b'_k}\sqcup\set{z}\sqcup\set{y_\ell}]$, and define $\lambda\in \TAB$ by 
$$ \lambda = \pmap{b'_k & z & y_\ell \\ b'_k & b_1 & c},$$
where $c$ can be taken in $\ClotX[\emptyset]$ whenever this is non-empty, and otherwise we can take $c$ to be any element in $\set{b'_k}\neq \emptyset$ since $s\geq e=1$ in that case.
Then $B\lambda\alpha = \clotX[\set{b'_k\alpha}]\neq \clotX[\set{b'_k\alpha}\sqcup\set{b_1\alpha}] = A\lambda\alpha$ so $\lambda\alpha\notin Q$ and $\dim(B\lambda\alpha) = \Card{K} = s$ as required.
\qed\end{proof}

We can now finally show the characterisation of the ideals in $\TAB$.

\begin{theorem}
The ideals of \TAB are precisely the sets $T_{r(S)} \cup K(S)$ and $T_{r(S)^+}\cup K(S)$ where $S$ is a non-empty subset of \TAB.
\end{theorem}
\begin{proof}
Taking into account Lemma \ref{Trs union K(S) are ideals}, it suffices to show that for any ideal $I$ of \TAB, there exists a set $S$ such that $I=T_{r(S)} \cup K(S)$ or $I=T_{r(S)^+} \cup K(S)$.
So let $I$ be an ideal of \TAB.

If $I$ is the smallest ideal of \TAB, that is, if $I=T_{e^+}$, then we set $S=T_{e^+}$.
From the discussion preceding Lemma \ref{Trs union K(S) are ideals}, we obtain $I = T_{r(S)}\cup K(S)$ as required.

\smallskip
From now on, let us assume that $I\neq T_{e^+}$, and let $\alpha\in I\setminus T_{e^+}$.
Then $\rho(\alpha)>e$ and $\alpha$ has the property that there exist $z_1,z_2\in\im\alpha\subseteq B$ such that $z_1\notin \clotX[z_2]$.
If we can find two such elements with one of $z_1$ and $z_2$ not in $B\alpha$, then we have that $A\alpha \neq B\alpha$ and thus $\alpha\notin Q$.
On the contrary, if for any two elements $z_1$ and $z_2$ satisfying the property above, both lie in $B\alpha$, then we get that $\im\alpha = B\alpha$ and thus $\alpha\in Q$.
In this second case, since $\rho(\alpha)\geq e+1$ we call upon Lemma \ref{non-reg map with smaller rank} with $s=e$ to get that that there exists $\lambda\in \TAB$ such that $\lambda\alpha\notin Q$.
Moreover, since $I$ is an ideal, we also have that $\lambda\alpha\in I$.
Therefore in both cases, if we take $S=I\priv Q$ we have that $S \neq \emptyset$. 
Now by setting $r=r(S)$, we show that $I$ is equal to either $T_r\cup K(S)$ or $T_{r^+} \cup K(S)$.

Firstly, we have that $T_r\cup K(S)\subseteq I$.
Indeed, let $\beta\in K(S)$, so that $\beta = \alpha\mu$ for some $\alpha\in S\subseteq I$ and $\mu\in\TAB$, and thus $\beta\in I$.
Similarly, let $\beta\in T_r$. If $\rho(\beta) >\dim (B\alpha)$ for all $\alpha\in I$, then, in particular $\rho(\beta) > \dim (B\alpha)$ for all $\alpha\in S$, which contradicts the minimality of $r$ since $\rho(\beta) < r$.
Therefore, there exists $\alpha\in I$ such that $\rho(\beta) \leq \dim (B\alpha)$. 
By Proposition \ref{GJ in TAB}, this means that $\beta = \lambda\alpha\mu$ for some $\lambda\in\TAB$, $\mu\in\TAB^1$ which shows that $\beta\in I$. Therefore $T_r\cup K(S)\subseteq I$.

It is also clear that $I\priv Q = S\subseteq K(S)$.
Similarly, for any $\gamma\in I\cap Q$ such that $\dim(B\gamma) < r$, we have that $r>\dim(B\gamma) = \rho(\gamma)$ and therefore $\gamma\in T_r$.
We now need to distinguish between the case where $I = T_r\cup K(S)$ and $I=T_{r^+}\cup K(S)$ by looking at the possible values of $\dim (B\gamma)$ for $\gamma\in I\cap Q$.

On one hand, if $\dim(B\gamma) < r$ for all $\gamma\in I\cap Q$, then, using the argument in the previous paragraph, we have shown that $I = T_r\cup K(S)$.

On the other hand, assume that there exists at least one $\beta\in I\cap Q$ such that $\dim(B\beta)\geq r$ and set $\kappa = \rho(\beta) = \dim(B\beta)$. 
We show in this case that $I = T_{r^+}\cup K(S)$.
If $\kappa > r$, then using Lemma \ref{non-reg map with smaller rank} with $s=r$, we get that there exists $\lambda\in \TAB$ such that $\dim(B\lambda\beta) = r$ and $\lambda\beta\notin Q$. But then $\lambda\beta\in I\priv Q=S$ and this contradicts the definition of $r = r(S)$. 
Therefore we must have that $\kappa = r$ from which we have that $\beta\in T_{r^+}$.
This gives us that $I\cap Q\subseteq T_{r^+}$ and therefore $I\subseteq T_{r^+}\cup K(S)$.
In order to get the equality we only need to show that $T_{r^+}\subseteq I$ since we already know that $T_r\cup K(S)\subseteq I$.
For this, consider $\gamma\in T_{r^+}$. 
Since in the current case there exists a map $\beta\in I$ with $\dim(B\beta) = r$, we have that $\rho(\gamma)\leq r = \dim(B\beta)$.
Then by Proposition \ref{GJ in TAB}, $\gamma = \lambda\beta\mu$ for some $\lambda\in \TAB, \mu\in \TAB^1$ which shows that $\gamma\in I$ and gives us that $T_{r^+}\subseteq I$. 
Therefore, when there exists a map in $I\cap Q$ which rank is equal to $r$, we have that $I = T_{r^+}\cup K(S)$.
\qed\end{proof}

\begin{remark}
Notice that for two distinct sets $S$ and $S'$, one can have $T_{r(S)}\cup K(S) = T_{r(S')}\cup K(S')$.
Similarly, for an ideal $I$, there might exist sets $S$ and $S'$ such that $I=T_{r(S)}\cup K(S)$ and $I = T_{r(S')^+}\cup K(S')$. 
Indeed, if we take $I = T_k$, then we know that $I = T_{r(S)}\cup K(S)$ for $S = T_k$, but we can also obtain $I$ using the construction of the proof with the set $S'=I\priv Q$.
More precisely, if $k$ is finite, then we have that $r(S') = k-1$, and there exist many elements $\gamma\in T_k$ for which $\dim (B\gamma) \geq k-1$, namely all idempotents on subalgebras of rank $k-1$. Thus we have that $T_k = I = T_{r(S')^+}\cup K(S')$.
On the other hand, if $k$ is infinite, then we directly have that $r(S')=k$ and that all elements of $I$ are such that $\dim (B\gamma)\leq \rho(\gamma)<k$, which gives us that $T_k = I=T_{r(S')}\cup K(S')$ in this case.
\end{remark}

Using the theorem above, we can give examples of the construction of two ideals that are not comparable under inclusion, as long as we are not in the case where the algebra $\AlgA$ is a set with $3$ elements and its subalgebra $\AlgB$ has dimension exactly $2$. 
In that specific case, the $8$ elements of $\TAB$ form three ideals $T_2$, $T_2\cup \set{\alpha\in \Qc\mid \rho(\alpha)=2}$ and $\TAB$ itself, which make up a chain.

\begin{example}
Let us assume that there exist independent elements $b_1, b_2\in B$, that $\AlgA$ does not have any constants and that $B$ has codimension at least $2$ in $A$. 
Let $B=\clotX[\set{b_1,b_2}\sqcup\set{d_j}]$ and $A=\clotX[\set{b_1,b_2}\sqcup\set{d_j}\sqcup\set{x_1,x_2}\sqcup\set{y_k}]$.
Define the following maps in \TAB:
$$ \alpha = \pmap{x_1 & \set{b_i,d_j,x_2,y_k} \\ b_1 & b_2} \quad \text{and}\quad 
\beta = \pmap{x_2 & \set{b_i,d_j,x_1,y_k} \\ b_1 & b_2}. $$
From this we have that $\rho(\alpha) = \rho(\beta)= 2$, $\dim(B\alpha)=\dim(B\beta)=1$ and $(x_1,b_1)\in\ker\beta\priv\ker\alpha$ while $(x_2,b_1)\in\ker\alpha\priv\ker\beta$, which shows that $\ker\alpha$ and $\ker\beta$ are incomparable.
Since $r(\set{\alpha})=r(\set{\beta})=2$, we let $I_\alpha = T_2 \cup K(\set{\alpha})$ and $I_\beta = T_2 \cup K(\set{\beta})$, which are both ideals from the previous theorem.
Now it is easy to see that $\alpha\in I_\alpha\priv I_\beta$ and $\beta\in I_\beta\priv I_\alpha$, which shows that these two ideals are incomparable.

A similar trick can be used with $b_2$ a constant, and another small modification can be used to find non comparable ideals when $B$ contains at least three independent elements, or two independent elements and a constant.
\end{example}

\section{Extended Green's relations}\label{section on extended Greens}

The study of extended Green's relations was introduced by Pastjin in \cite{Pj} and extended by El-Qallali \cite{ElQallali} and Lawson \cite{Lawson}, and, since then, they have been widely used in order to generalise the notion of regular semigroups to abundant, Fountain and $U$-semiabundant semigroups.
For $a, b\in S$, these relations are given as follows:
$$
\begin{gathered}
a\GLs b \Longleftrightarrow \left( ax=ay \Leftrightarrow bx = by \quad \forall x,y\in S^1 \right), \\
a\GRs b \Longleftrightarrow \left( xa=ya \Leftrightarrow xb = yb \quad \forall x,y\in S^1  \right), \\
a\GLt b \Longleftrightarrow \Big( af=a \Leftrightarrow bf = b \quad \forall f\in E(S) \Big), \\
a\GRt b \Longleftrightarrow \Big(fa=a \Leftrightarrow fb = b \quad \forall f\in E(S) \Big),\\
\GHs = \GLs \wedge \GRs, \quad \phantom{and} \quad \GHt = \GLt \wedge \GRt, \\
\GDs = \GLs \vee \GRs, \quad \text{and} \quad \GDt = \GLt \vee \GRt.
\end{gathered}
$$
Of course, there are analogues extended relations $\GJs$ and $\GJt$ of Green's relation $\GJ$ but we delay their description which is less straightforward until these relations are needed.
Also, the description of the relations $\GHs$ and $\GHt$ will not be given as they can be easily deduced from the descriptions of the appropriate $*$ and $\sim$-relations.
A well-known consequence of these definitions (see \cite{Fountain-abundant}) is that two elements are $\GRs$ related in a semigroup $S$ if and only if they are $\GR$ related in an oversemigroup $T$, and dually for $\GLs$.
Also it is easy to see that by definition, $\GLs$ and $\GRs$ are right, respectively left, congruences and that we have the inclusions $\GR \subseteq \GRs \subseteq \GRt$ and $\GL \subseteq \GLs \subseteq \GLt$.
These inclusions may be strict but it is well-known that they become equalities when we restrict our attention to regular elements as given in the next lemma (the proof of which is only included for completeness purposes).

\begin{lemma}[\cite{ElQallali}]\label{greens in reg smgrp}
In a semigroup $S$, if $a$ and $b$ are both regular elements, then $a\GR b\Leftrightarrow a\GRt b$ and $a\GL b\Leftrightarrow a\GLt b$.
\end{lemma}
\begin{proof}
By the remark above, we only need to prove the implication from right to left.
So assume that $a$ and $b$ are regular elements such that $a\GRt b$. 
Then there exist $a', b'\in S$ such that $a = aa'a$ and $b = bb'b$.
Thus $aa'$ and $bb'$ are idempotents and by the definition of being $\GRt$ related, we get that $a = bb'a$ and $b = aa'b$, and thus $a\GR b$.
A dual argument shows that the similar result holds for $\GLt$.
\qed\end{proof}

As a direct consequence of this lemma, we have that the relations $\GRQ$, $\GRs[Q]$, and $\GRt[Q]$ are equal (and similarly $\GLQ=\GLs[Q]=\GLt[Q]$), but this is not true on the whole of $\TAB$. 
Indeed, the following propositions shows that on $\TAB$ the relations $\GLs$ and $\GLt$ are equal but they differ from $\GL$, while the relations $\GR$, $\GRs$ and $\GRt$ are all distinct.

\begin{proposition}\label{GLt in TAB}
In $\TAB$ we have that $\alpha\GLt\beta$ if and only if $\im\alpha = \im\beta$. Thus $\GLs = \GLt$ in $\TAB$.
\end{proposition}
\begin{proof}
Let $\alpha \in \TAB$ and $\eta\in E$. 
If $\alpha\eta = \alpha$, then it is clear that $\im\alpha \subseteq \im\eta$.
Conversely, assume that $\im\alpha \subseteq\im\eta$. Since $\eta$ is an idempotent, it follows that $\eta\lvert_{\im\eta} = \id_{\im\eta}$ and then $\eta\lvert_{\im\alpha} = \id_{\im\alpha}$ which shows that $\alpha\eta = \alpha$.
Therefore for any $\alpha\in\TAB$, $\eta\in E$, we have that $\alpha\eta = \alpha$ if and only if $\im\alpha\subseteq\im\eta$.

Let $\alpha,\beta \in \TAB$ and assume that $\alpha\GLt\beta$. 
Now let $\im\alpha = \clotX[\set{x_i\alpha}]$, $\im\beta = \clotX[\set{y_k\beta}]$ and $A = \clotX[\{x_i\alpha\}\sqcup \{x_j\}] = \clotX[\{y_k\beta\}\sqcup \{y_\ell\}]$. Define idempotents $\eta$ and $\theta$ as follows:
$$\eta = \pmap{x_i\alpha & x_j\\ x_i\alpha & \tilde{x}} \qquad \theta = \pmap{y_k\beta & y_\ell \\ y_k\beta & \tilde{y}},$$
for some elements $\tilde{x}\in \clotX[x_i\alpha]$ and $\tilde{y}\in\clotX[y_k\beta]$.
It is clear that $\im\eta = \im\alpha$ and $\im\theta = \im\beta$ which gives us that $\alpha\eta = \alpha$ and $\beta\theta = \beta$.
Since $\alpha\GLt\beta$, it follows that $\alpha\theta = \alpha$ and $\beta\eta = \beta$. 
Hence, by the equivalence previously proved, we have that $\im\alpha \subseteq \im\theta = \im\beta$ and $\im\beta \subseteq \im\eta = \im\alpha$.
Therefore $\im\alpha = \im\beta$.
Conversely, it can easily be seen that for any $\alpha, \beta\in \TAB$ with $\im\alpha = \im\beta$, any idempotent $\eta$ satisfying $\alpha\eta = \alpha$ will be such that $\im\alpha \subseteq \im\eta$ and thus $\im\beta\subseteq \im\eta$ and $\beta\eta = \beta$, so $\alpha\GLt\beta$.

Finally, since $\alpha\GLE\beta$ implies $\alpha\GLs\beta$, we get from the description of $\GLE$ the following inclusions:
$$\{(\alpha,\beta)\mid \im\alpha = \im\beta\} \subseteq \GLs\subseteq \GLt = \{(\alpha,\beta)\mid \im\alpha = \im\beta\}, $$
which forces the equalities, and so $\GLs = \GLt$.
\qed\end{proof}

Whereas the description of $\GLs$ and $\GLt$ follows from the description of $\GLQ$, the same cannot be said for $\GRs$ and $\GRt$ and their explicit descriptions are given by Propositions \ref{GRs in TAB} and \ref{GRt in TAB}.

\begin{proposition}\label{GRs in TAB}
In $\TAB$, we have that $\alpha\GRs \beta$ if and only if one of the following occurs:
\begin{enumerate}
    \item $\ker\alpha = \ker\beta$ and $\alpha, \beta\in Q$;
    \item $\ker\alpha \cap (B\times B) = \ker\beta \cap (B\times B)$ and $\alpha,\beta \notin Q$.
\end{enumerate}
\end{proposition}
\begin{proof}
Let $\alpha\in Q$, $\beta\in \TAB$ and assume $\alpha\GRs\beta$. 
From Lemma \ref{Regular maps factorised by idp}, we know that there exists $\gamma\in \TAB$ such that $\alpha=\gamma\alpha$ and therefore $\beta = \gamma\beta$, so $\gamma$ is a left-identity for $\beta$ and thus from Lemma \ref{Regular maps factorised by idp} again, we have that $\beta\in Q$.
Since the case for $Q\times Q$ is given by Lemma \ref{greens in reg smgrp}, we only need to consider the case of two non-regular elements.

Notice that for any map $\alpha\in\TAB$, if $\gamma\alpha = \delta\alpha$ with $\gamma \neq \delta = 1$, it follows that $A\alpha=A\gamma\alpha \subseteq B\alpha$ and then $\alpha\in Q$.
Therefore in what follows, either $\gamma = \delta = 1$ or $\gamma, \delta \neq 1$.

Now, let $\alpha,\beta\in \Qc$ and assume that $\alpha\GRs\beta$, that is, $\gamma\alpha = \delta\alpha$ if and only if $\gamma\beta = \delta\beta$ for $\gamma, \delta \in \TAB^1$. 
Then for any pair $(b_1, b_2)\in \ker\alpha\cap (B\times B)$ with $b_1\neq b_2$, we construct specific maps $\gamma, \delta\in \TAB$ satisfying the relation $\gamma\alpha = \delta\alpha$. 
To this end, let $B = \clotX[\set{y_k}]$ and $A = \clotX[\{y_k\}\sqcup \{u\}\sqcup \{x_j\}]$ and define $\gamma, \delta \in \TAB$ by:
$$ \gamma = \pmap{y_k & u & x_j \\ y_k & b_1 & y_1}  \quad \text{and} \quad 
\delta = \pmap{y_k & u & x_j \\ y_k & b_2 & y_1}.$$
It is clear that $\gamma\alpha = \delta\alpha$, and since $\alpha\GRs\beta$, it follows that $\gamma\beta = \delta\beta$.
Therefore $b_1\beta = u\gamma\beta = u\delta\beta = b_2\beta$ and thus $(b_1,b_2)\in \ker\beta\cap(B\times B)$.
Using the same argument interchanging the roles of $\alpha$ and $\beta$, we deduce that $\ker\alpha \cap (B\times B) = \ker\beta \cap (B\times B)$.

Conversely, assume that $\ker\alpha \cap (B\times B) = \ker\beta \cap (B\times B)$ and that $\gamma\alpha = \delta\alpha$ for some $\gamma, \delta\in \TAB$. 
Then for any $a\in A$ we have $a\gamma\alpha = a\delta\alpha$ from which $(a\gamma, a\delta)\in \ker\alpha$. 
But since $a\gamma, a\delta\in B$, it follows that $(a\gamma, a\delta)\in \ker\beta$. 
Therefore $a\gamma\beta = a\delta\beta$, and since this is true for any $a\in A$ we get that $\gamma\beta = \delta\beta$.
Interchanging the roles of $\alpha$ and $\beta$, we obtain that $\gamma\alpha = \delta\alpha$ if and only if $\gamma\beta = \delta\beta$ for any $\gamma, \delta\in \TAB$. 
Since this equivalence also holds when $\gamma = \delta = 1$ together with the note made earlier in the proof, we conclude that $\alpha\GRs\beta$ as required.
\qed\end{proof}

\begin{proposition}\label{GRt in TAB}
In $\TAB$, we have that $\alpha\GRt\beta$ if and only if $\alpha, \beta\in \Qc$ or ($\ker\alpha = \ker\beta$ and $\alpha,\beta \in Q$).
\end{proposition}
\begin{proof}
From Lemma \ref{Regular maps factorised by idp} we know that $\alpha = \eta\alpha$ where $\eta\in E$ if and only if $\alpha$ is regular.
Using this and Lemma \ref{greens in reg smgrp}, the proposition follows easily, since regular maps can only be $\GRt$ related to regular maps, and for two non-regular maps the condition in the definition of $\GRt$ is trivially realised.
\qed\end{proof}

In order to compare kernels the following lemma will come in handy:
\begin{lemma}\label{lem kernel equalities by map description}
Let $B = \clotX[\set{b_k}\sqcup\set{b_i}]$ and $A=B\sqcup \clotX[\set{a_j}]$. Suppose that $\gamma, \delta\in\TAB$ can be expressed as follows:
$$ \gamma = \pmap{b_k & b_i & a_j \\ f_k & t_i(\ob{f_k}) & g_j} \qquad \text{and} \qquad \delta = \pmap{b_k & b_i & a_j \\ x_k & t_i(\ob{x_k}) & y_j},$$
for some terms $t_i$, where the sets $\set{f_k}$ and $\set{x_k}$ are respective bases of $\im(B\gamma)$ and $\im(B\delta)$.
Then we have that $\ker\gamma \cap (B\times B) = \ker\delta\cap (B\times B)$. Consequently, if $\gamma, \delta\in\Qc$ we also obtain that $\gamma\GRs\delta$.
\end{lemma}
\begin{proof}
Let us assume that the maps $\gamma, \delta\in\TAB$ are given as above. 
Let $(a, b)\in \ker\gamma\cap (B\times B)$. Then there exist terms $u$ and $v$ such that $a = u(\ob{b_k}, \ob{b_i})$ and $b = v(\ob{b_k}, \ob{b_i})$. 
Then we have that $u(\ob{f_k}, \ob{t_i(\ob{f_k})}) = a\gamma = b\gamma = v(\ob{f_k}, \ob{t_i(\ob{f_k})})$. 
Now we extend $\set{f_k}$ to a basis of $A$ through $\set{d_m}$, and we define $\eta =  \pmap{f_k & d_m \\ x_k & x_1}\in \TAB$.
From this we get the following:
\begin{align*}
a\delta &= u(\ob{b_k}, \ob{b_i})\delta = u(\ob{x_k}, \ob{t_i(\ob{x_k}}) 
= u(\ob{f_k\eta}, \ob{t_i(\ob{f_k\eta})}) = u(\ob{f_k}, \ob{t_i(\ob{f_k})})\eta = a\gamma\eta \\
&= b\gamma\eta =  v(\ob{f_k}, \ob{t_i(\ob{f_k})})\eta = v(\ob{x_k}, \ob{t_i(\ob{x_k}}) 
= b\delta,
\end{align*}
which means that $(a,b)\in\ker\beta\cap (B\times B)$. 
Hence $\ker\gamma\cap(B\times B)\subseteq \ker\delta\cap (B\times B)$.
The converse inclusion works similarly by using the map $\theta =  \pmap{x_k & e_m \\ f_k & f_1}\in \TAB$.
Therefore we have that $\ker\gamma \cap (B\times B) = \ker\delta\cap (B\times B)$. 
Moreover, if $\gamma, \delta\in\Qc$ then Proposition \ref{GRs in TAB} allows us to conclude that $\gamma\GRs\delta$.
\qed
\end{proof}

Recall from the introduction that a semigroup is called \emph{left abundant} [resp. \emph{left Fountain}] if each $\GRs$-class [resp. each $\GRt$-class] contains an idempotent, and dually for the right-handed notion using $\GLs$ and $\GLt$.
From the description of our extended Green's relations, we can see that $\TAB$ is right abundant (and thus right Fountain) but is not left abundant nor left Fountain since no idempotent lies in the $\GRs$ or $\GRt$-class of an element in $\Qc$.

As noticed at the beginning of this section $\GLs$ and $\GRs$ are respectively right and left congruences, and thus $\GLt$ is also a right congruence. 
However, it is not always true that the same holds for $\GRt$, as is shown by the following lemma.

\begin{lemma}\label{GRt congruence cases}
The equivalence $\GRt$ is a left congruence if and only if one of the following occurs:
\begin{itemize}
    \item $\dim B = 2$ and one-dimensional subalgebras are singletons; or
    \item $\dim B=1$ and $\left\lvert \Csts\right\rvert = 1$.
\end{itemize}
\end{lemma}
\begin{proof}
First, notice that if $\alpha, \beta\in \TAB$ are such that $\ker\alpha\cap (B\times B)=\ker\beta\cap (B\times B)$, then for any $\gamma\in \TAB$, Lemma \ref{condition product to be regular} tells us that $\gamma\alpha\in Q$ if and only if $\gamma\beta\in Q$, and we also have that $(x\gamma, y\gamma)\in \ker\alpha$ if and only if $(x\gamma, y\gamma)\in \ker\beta$ (as $x\gamma,y\gamma\in B$) showing that $\ker\gamma\alpha=\ker\gamma\beta$.
From this we see that either $\gamma\alpha$ and $\gamma\beta$ are in $Q$ and have the same kernel or neither of them is in $Q$, which gives us in both cases that they are $\GRt$-related.
Since $\GRt[Q] = \GRQ$ and $Q$ is a union of $\GRt$-classes, it follows that $\GRt$ when restricted to $Q\times Q$ is a left congruence.
Therefore $\GRt$ can only fail to be a left congruence on $\Qc\times\Qc$. 
To this extend, we need to find maps $\alpha,\beta\in \Qc$ (which are then $\GRt$-related by Proposition \ref{GRt in TAB}) and $\gamma\in\TAB$ such that $\gamma\alpha$ and $\gamma\beta$ are not $\GRt$-related. 
That is, we need either that only one of $\gamma\alpha$ and $\gamma\beta$ lie in $Q$, or that both products are regular but have a different kernel.

With this in mind, let us assume that $\dim B=2$ and that one-dimensional subalgebras are singletons and let $\alpha,\beta\notin Q$. 
Then $B\alpha\subsetneq A\alpha \subseteq B$ and similarly for $\beta$ so $\rho(\alpha\lvert_B)=\rho(\beta\lvert_B)=1$. 
Also, since one-dimensional subalegras are singleton, this forces $B\alpha=\set{c}$ and $B\beta=\set{c'}$ for some $c,c'\in B$ and thus $\ker\alpha\cap (B\times B)=\ker\beta\cap (B\times B)$. 
Using the previous argument, it follows that for any map $\gamma\in \TAB$, we have that $\gamma\alpha\GRt\gamma\beta$, and thus $\GRt$ is a left congruence in this case.

Similarly, if $\dim B=1$ and the constant subalgebra only consists of a single element, say $0$, then it is clear that for any map $\alpha\notin Q$, we have that $B\alpha = \set{0}$, and thus the argument above works and $\GRt$ is a left congruence in this case. 

\medskip
In order to show that outside of these cases $\GRt$ fails to be a left congruence, we exhibit counterexamples.
Let us assume first that $\dim B\geq 3$ and let $B=\clotX[ \set{y_m}\sqcup \set{b_i}]$, $A=\clotX[\set{x}\sqcup \set{a_j}\sqcup \set{y_m}\sqcup \set{b_i}]$ where $\Card{M}=3$ and the sets $\set{a_j}$ and $\set{b_i}$ are possibly empty.
Consider the following maps:
$$ \alpha = \pmap{x & y_1 & y_2 & y_3 & a_j, b_i \\ y_1 & y_3 & y_3 & y_3 & y_3} \quad \text{and} \quad \beta = \pmap{x & y_1 & y_2 & y_3 & a_j, b_i \\ y_1 & y_2 & y_2 & y_3 & y_3}. $$
Then it is clear that $\alpha, \beta\notin Q$ and thus they are $\GRt$-related. 
Also, we have that $\alpha^2 \in Q$ but $\alpha\beta\notin Q$, and thus they cannot be related. 

Now assume that $\dim B=2$ with $B=\clotX[\set{b_1,b_2}]$, $A=\clotX[\set{a_j}\sqcup \set{ b_1, b_2}]$ and suppose that one-dimensional subalgebras are not singletons, that is, that there exists a term $g$ such that $g(b_2)\neq b_2$.
Then for the maps
$$ \alpha = \pmap{a_j & b_1 & b_2 \\ b_1 & b_2 & b_2}, \quad \beta = \pmap{a_j & b_1 & b_2 \\ b_1 & g(b_2) & b_2}, \text{ and } \gamma = \pmap{a_j & b_1 & b_2 \\ b_1 & b_1 & b_2},$$
we have the following:
\begin{itemize}
    \item $\alpha, \beta\notin Q$ and thus $\alpha\GRt\beta$;
    \item $\gamma\alpha = \pmap{a_j & b_1 & b_2 \\ b_2 & b_2 & b_2}$ and thus $\gamma\alpha\in Q$;
    \item $\gamma\beta = \pmap{a_j & b_1 & b_2 \\ g(b_2) & g(b_2) & b_2}$ and thus $\gamma\beta\in Q$;
    \item $(b_1,b_2)\in \ker\gamma\alpha$, but $(b_1,b_2)\notin \ker\gamma\beta$.
\end{itemize} 
Therefore we have that $\gamma\alpha$ and $\gamma\beta$ are not $\GRt$-related, whereas $\alpha$ and $\beta$ are.
Similarly, if $\dim B=1$ and $\ClotX[\emptyset]$ contains at least two distinct elements, then by abusing notation on the counterexample where $\dim B=2$, we can take $b_2\in \ClotX[\emptyset]$ and choose the term $g$ such that $g(b_2)\neq b_2$.
This shows that $\GRt$ is not a left congruence in these two cases.
\qed\end{proof} 

In any semigroup we can define the relation $\GDs$ by being the join of the relations $\GLs$ and $\GRs$.
A well-known characterisation (see for example in Howie \cite[Prop 1.5.11]{Howie}) tells us that $a\GDs b$ if and only if for some $n\in \mathbb{N}$ there exist $c_0,c_1, \dots, c_{2n}\in S$ such that $a=c_0\GLs c_1\GRs c_2\dots \GLs c_{2n-1}\GRs c_{2n}=b$.
If $\GLs$ and $\GRs$ were to commute, this sequence of compositions would result in a single composition.
However Pastjin \cite{Pj} noticed that in general this is not the case, as can be seen here by the following example.

\begin{example}
Let $A = \clotX[\set{x_1,x_2,x_3,x_4,x_5}]$ and $B=\clotX[\set{x_3,x_4,x_5}]$.
Define $\alpha, \beta,\gamma\in \TAB$ by the following:
$$\begin{gathered} 
\alpha = \pmap{x_1 & x_2 & x_3 & x_4 & x_5 \\ x_3 & x_3 & x_3 & x_4 & x_5}, \qquad 
\beta = \pmap{x_1 & x_2 & x_3 & x_4 & x_5 \\ x_3 & x_3 & x_5 & x_5 & x_5}, \\
\gamma = \pmap{x_1 & x_2 & x_3 & x_4 & x_5 \\ x_3 & x_4 & x_5 & x_5 & x_5}.
\end{gathered}$$
Then we clearly have that $\alpha\in Q$, $\beta\notin Q$, $\gamma\notin Q$ and from Propositions \ref{GLt in TAB} and \ref{GRs in TAB} these maps satisfy the relations $\im\alpha = \im\gamma$ and $\ker\beta \cap (B\times B) = \ker\gamma\cap(B\times B)$.
Therefore $\alpha \GLscGRs \beta$ through $\gamma$.

Now, in order to have $\alpha\GRscGLs\beta$, we need to find $\delta\in \TAB$ such that $\im\delta = \im\beta = \clotX[\set{x_3,x_5}]$ and $\ker\alpha = \ker\delta$ (since $\alpha\in Q$).
However, for any such $\delta$ we would need 
$$\im\delta \cong A/\ker\delta = A/\ker\alpha \cong\im\alpha,$$ 
giving that $\dim(\im\delta) = 3$ which is impossible.
Therefore, no such $\delta\in \TAB$ can exist and $\GLs$ and $\GRs$ do not commute in $\TAB$.
\end{example}

When replacing in the example above the relation $\GRs$ by $\GRt$, the same arguments hold since $\gamma,\beta\notin Q$ implies that $\gamma\GRt\beta$.
Therefore, similarly to $\GLs$ and $\GRs$, we can see that the relations $\GLt=\GLs$ and $\GRt$ do not commute either in $\TAB$, exhibiting a different behaviour from the usual Green's relations.

Nevertheless, in the case of $\TAB$ it is possible to give a precise characterisation for $\GDs$ which, surprisingly, depends on the corank of the subalgebra $\AlgB$ inside $\AlgA$. 
Some elements are easily seen to be $\GDs$-related, such as those of the same rank, as given by the following lemma.

\begin{lemma}\label{lem maps of same rank are D* rel}
Let $\alpha, \beta\in \TAB$ be such that $\rho(\alpha) = \rho(\beta)$. Then $\alpha\GLscGRs\beta$ and thus $\alpha\GDs\beta$.
\end{lemma}

\begin{proof}
For any $\alpha,\beta\in \TAB$ with $\rho(\alpha) = \rho(\beta)$, Lemmas \ref{subalgebras of same rank are isomorphic} and \ref{isom of images gives RL related maps} give us that there exists $\mu\in\TAB$ with $\im\alpha = \im\mu$ and $\ker\beta = \ker\mu$ (so that also $\ker\beta\cap(B\times B) = \ker\mu\cap (B\times B)$).
Thus $\alpha\GLs\mu$ by Proposition \ref{GLt in TAB}.
Additionally, using the appropriate case on whether $\beta$ and $\mu$ are regular in Proposition \ref{GRs in TAB}, we also have that $\mu\GRs\beta$.
Therefore we have $\alpha\GLscGRs\beta$ and $\alpha\GDs\beta$.
\end{proof}

However, the full characterisation of $\GDs$ requires us to concentrate on the composition $\GLscGRs$ beforehand.

\begin{proposition}\label{LoRstars in TAB}
Let $\alpha\in \TAB$. Then $\alpha\GLscGRs\beta$ for some $\beta\in \TAB$ if and only if one of the following happens:
\begin{enumerate}
    \item $\beta\in Q$ and $\im\alpha \cong\im\beta$;
    \item $\beta\notin Q$ and $\im(B\beta) \cong \im\alpha$ with $\rho(\beta\lvert_B)\geq \aleph_0$;
    \item $\beta\notin Q$ and $\rho(\beta\lvert_B) < \rho(\alpha) \leq \rho(\beta\lvert_{B})+\codim[A] B$.
\end{enumerate}
Exchanging the roles of $\alpha$ and $\beta$, we have the dual characterisation for when $\alpha\GRscGLs\beta$.
\end{proposition}

\begin{proof}
Let $\alpha\in\TAB$ and suppose that $\beta\in\TAB$ is such that $\alpha\GLscGRs\beta$. 
Then there exists $\gamma\in\TAB$ such that $\alpha\GLs\gamma$ and $\gamma\GRs\beta$. Therefore we always have that $\im\alpha =\im\gamma$ by Proposition \ref{GLt in TAB}, and Proposition \ref{GRs in TAB} tells us that either $\beta, \gamma\in Q$ are such that $\ker\gamma = \ker\beta$, or $\beta, \gamma\notin Q$ and then $\ker\gamma\cap(B\times B)=\ker\beta \cap(B\times B)$.
If $\beta\in Q$ (and hence also $\gamma\in Q$), then we have that 
$$\im\alpha = \im\gamma \cong A/\ker\gamma = A/\ker\beta \cong \im\beta,$$
which gives us the first case.

We now assume that $\beta,\gamma\notin Q$.
Thus $$\im\beta\lvert_B\cong B/\left(\ker\beta\cap (B\times B)\right) = B/\left(\ker\gamma\cap (B\times B)\right) \cong \im\gamma\lvert_B,$$ so that $\rho(\beta\lvert_B) = \rho(\gamma\lvert_B)$.
If $\rho(\gamma)=\rho(\gamma\lvert_B)$, then $\rho(\gamma)\geq \aleph_0$ by Lemma \ref{lem rank of restriction wrt finiteness} and from $\rho(\beta\lvert_B) = \rho(\gamma) = \rho(\alpha)$ we get that $\im(B\beta) \cong \im \alpha$ which corresponds to the second case.
Otherwise, $\rho(\gamma\lvert_B)<\rho(\gamma)$ and then, setting $Z$ to be a basis extension of $B$ in $A$, we have that
\begin{align*} 
\rho(\beta\lvert_B) = \rho(\gamma\lvert_B) &< \rho(\alpha) = \rho(\gamma) = \dim (A\gamma) 
\leq \dim (B\gamma) + \dim (Z\gamma) \\ 
&\leq \rho(\gamma\lvert_B) + \dim Z = \rho(\beta\lvert_B) + \codim[A] B,
\end{align*}
giving us the remaining case.

\medskip
For the converse, let $\alpha\in\TAB$. 
We show that each case gives us that $\alpha\GLscGRs\beta$.
\begin{enumerate}
\item If $\im\alpha\cong\im\beta$ and $\beta\in Q$, then Lemma \ref{lem maps of same rank are D* rel} directly gives us that $\alpha\GLscGRs\beta$.
\item Suppose now that $\beta\notin Q$ is such that $\im(B\beta)\cong \im\alpha$ with $\rho(\beta\lvert_B)\geq \aleph_0$. 
We then write $B\beta = \clotX[\set{b_k\beta}]$, $B=\clotX[\set{b_k}\sqcup \set{b_i}]$ and $A = \clotX[\set{b_k}\sqcup \set{b_i}\sqcup\set{a_j}]$. 
By assumption, we have that $A\alpha = \clotX[\set{x_k\alpha}]$ for some $\set{x_k}\subseteq A$. Since $\Card{K} = \rho(\beta\lvert_B) \geq \aleph_0$ we have that by setting $K'= K\priv\set{1}$ there exists a bijection $\phi\colon K\to K'$. For all $k\in K$ we set $z_k = x_{k\phi}\alpha$ and we now define $\gamma\in\TAB$ by:
$$ \gamma = \pmap{b_k & b_i & a_j\\ z_k & t_i(\ob{z_k}) & x_1\alpha},$$ 
where the terms $t_i$ are such that $b_i\beta = t_i(\ob{b_k\beta})$. 
Then, we have that $\im\alpha = \clotX[\set{x_k\alpha}] = \im\gamma \neq \clotX[\set{z_k}] = \im(B\gamma)$ and thus $\alpha\GLs\gamma$ and $\gamma\notin Q$. 
Also, looking at the expression of $\gamma$ and noting that $\beta= \pmap{b_k & b_i & a_j\\b_k\beta & t_i(\ob{b_k\beta}) & a_j\beta}$, we have that $\gamma\GRs\beta$ by Lemma \ref{lem kernel equalities by map description}. 
Therefore $\alpha\GLscGRs \beta$ as expected.

\item Last, assume that $\beta\notin Q$ and $\rho(\beta\lvert_B) < \rho(\alpha) \leq \rho(\beta\lvert_B) + \codim[A] B$. Keeping the same notation as above for $B\beta$, $B$ and $A$, we write $A\alpha = \clotX[\set{x_k\alpha}\sqcup\set{y_\ell\alpha}]$ for some $\set{x_k}, \set{y_\ell}\subseteq A$ where $L\neq\emptyset$ (since $\rho(\alpha) > \rho(\beta\lvert_B)$). Since $\codim[A]B = \Card{J}$, we can now rewrite the equation on the ranks in terms of the underlying sets, which gives us that
$$\Card{K} < \Card{K\sqcup L} \leq \Card{K\sqcup J}.$$
Notice that if $\dim(B\beta) = \Card{K}$ is finite, then we directly obtain that $\Card{L}\leq \Card{J}$. 
Otherwise $\Card{K}$ is infinite and from the equation above we have that $\Card{K}<\Card{K\sqcup L} = \max\set{\Card{K}, \Card{L}}$, so that $\Card{L} > \Card{K}$.
Similarly $\Card{J}>\Card{K}$, which put together gives us that $\Card{L} =\Card{K\sqcup L} \leq \Card{K\sqcup J} = \Card{J}$.
From the inequality on the cardinals of $L$ and $J$ in both the finite and the infinite dimensional cases, we can extract a subset $J'\subseteq J$ such that $\phi \colon J'\to L$ is a bijection. From this we define $\set{z_j}\subseteq B$ for all $j\in J$ by $z_j = y_{j\phi}\alpha$ if $j\in J'$ and $z_j = y_1\alpha$ (which necessarily exists) otherwise. 
Under this definition, it is easy to see that $\clotX[z_j] = \clotX[\set{y_\ell\alpha}]$.
Now we define a map $\gamma\in\TAB$ as follows:
$$ \gamma = \pmap{b_k & b_i & a_j \\x_k\alpha & t_i(\ob{x_k\alpha}) & z_j},$$ 
where the terms $t_i$ are such that $b_i\beta = t_i(\ob{b_k\beta})$. Then we have that $\im\gamma = \clotX[\set{x_k\alpha}\sqcup\set{y_\ell\alpha}] = \im\alpha$, that is, $\alpha\GLs\gamma$. Similarly as above, we can also see that $\gamma\notin Q$ and that $\gamma\GRs\beta$ by Lemma \ref{lem kernel equalities by map description}.
Therefore $\alpha\GLscGRs\beta$, which concludes the last case of the proof.\qed
\end{enumerate}
\end{proof}

If $\AlgB$ is a maximal proper subalgebra of a $\AlgA$, then the set of maps that can be reached through a series of composition of $\GLs$ and $\GRs$ starting from a map $\alpha$ with finite rank is restricted to those having the same rank as $\alpha$. 
This is stated formally and proved in the following lemma.

\begin{lemma}\label{GLscGRs under dab 1}
Assume that $\codim[A] B=1$ and let $\alpha\in\TAB$ be a map of finite rank.
Then $\alpha\GLscGRs\beta$ for some $\beta\in\TAB$ if and only if $\rho(\alpha)=\rho(\beta)$.
\end{lemma}
\begin{proof}
Since the sufficient condition was given by Lemma \ref{lem maps of same rank are D* rel}, we only need to show one direction.
So we let $\beta\in\TAB$ be such that $\alpha\GLscGRs\beta$.

We first suppose that $\beta\in Q$. 
Then by Proposition \ref{LoRstars in TAB}, $\alpha\GLscGRs\beta$ if and only if $\im\alpha\cong\im\beta$ which is equivalent to $\rho(\alpha)=\rho(\beta)$.

Now suppose that $\beta\notin Q$. Since $\alpha$ has finite rank, only case 3 of Proposition \ref{LoRstars in TAB} can occur. Hence we have that $\rho(\beta\lvert_B)< \rho(\alpha)\leq\rho(\beta\lvert_B) + \codim[A] B = \rho(\beta\lvert_B) + 1$.
This forces $\rho(\beta\lvert_B)$ to be finite and using the contrapositive of Lemma \ref{lem rank of restriction wrt finiteness} we get
\begin{align*}
    \rho(\beta\lvert_B) =\dim(B\beta) &<\rho(\beta)=\dim(A\beta)\\
    &\leq \dim(B\beta) + \codim[A] B \\
    &= \rho(\beta\lvert_B)+1,
\end{align*}
which in turn forces $\rho(\beta)=\rho(\beta\lvert_B)+1$.
Similarly, $\rho(\alpha)=\rho(\beta\lvert_B)+1$, and so $\rho(\alpha)=\rho(\beta)$, which concludes the proof.
\qed\end{proof}

Conversely, if $\AlgB$ is not a maximal proper subalgebra then, by consecutive compositions of $\GLs$ and $\GRs$, we are able to go up and down the finite ranks as long as the map we started with does not have the minimal rank $e$ where $e$ was defined in Section \ref{section on Q} as the smallest rank of a subalgebra of $\AlgA$.
This process is given formally by the following lemma.

\begin{lemma}\label{GLscGRs under dab>1}
Assume that $\codim[A] B\geq 2$ and let $\alpha\in Q$ be such that $e<\rho(\alpha)<\aleph_0$.

If $\rho(\alpha)\geq e+2$ then there exist $\delta_1\notin Q$ and $\gamma_1\in Q$ such that $\rho(\delta_1)=\rho(\gamma_1)=\rho(\alpha)-1$ and $\alpha\GLscGRs\delta_1\GLs\gamma_1$.

If $\rho(\alpha) <\dim B$, then there exist $\delta_2\notin Q$ and $\gamma_2\in Q$ such that $\rho(\delta_2)=\rho(\alpha)$, $\rho(\gamma_2)=\rho(\alpha)+1$ and $\alpha\GLs\delta_2\GRscGLs\gamma_2$.

Consequently, for all $\beta\in \TAB$ such that $e<\rho(\beta)<\aleph_0$, there exists $n\in\N$ such that $\alpha\left(\GLscGRs\right)^n\beta$.
\end{lemma}
\begin{proof}
Suppose that $\codim[A]B\geq 2$ and consider $\alpha\in Q$ such that $e < \rho(\alpha)<\aleph_0$.
Since $\alpha$ is regular, we let $A\alpha=B\alpha=\clotX[\set{b_i\alpha}]$, $B=\clotX[\set{b_i}\sqcup\set{c_k}]$ and $A=\clotX[\set{b_i}\sqcup\set{c_k}\sqcup\set{x_j}]$ with $\Card{J}\geq 2$.
Then we have that 
$$\alpha=\pmap{b_i & c_k & x_j\\b_i\alpha & u_k(\ob{b_i\alpha}) & v_j(\ob{b_i\alpha})}$$
for some terms $u_k$ and $v_j$.

For the first part, since $\Card{I}=\rho(\alpha) \geq 2$, define $\delta_1$ and $\gamma_1$ as follows:
$$ \delta_1 = \pmap{b_{i\geq 3} & \set{b_1,b_2} & c_k & x_j\\b_i\alpha & d & d & b_2\alpha} \quad \text{and}
\quad \gamma_1 = \pmap{b_{i\geq 2} & b_1 & c_k & x_j\\b_i\alpha & b_2\alpha & b_2\alpha & b_2\alpha},$$
where the set $\set{b_{i\geq 3}}$ is possibly empty and the element $d\in B$ is taken as a constant if $e=0$, and $d=b_3\alpha$ otherwise (which necessarily exists since $\Card{I}=\rho(\alpha) \geq 3$ in that case).
Thus $\delta_1\notin Q$, $\gamma_1\in Q$ and we have that $\im\delta_1 = \clotX[\set{b_i\alpha}\priv \set{b_1\alpha}] = \im \gamma_1$, so that $\delta_1\GLs\gamma_1$ and $\rho(\delta_1)=\rho(\gamma_1)=\Card{I}-1$.
Also $\rho(\delta_1\lvert_B) = \Card{I}-2 < \Card{I} = \rho(\alpha)$ and $\rho(\delta_1\lvert_B) + \codim[A] B \geq \Card{I}-2 + 2 = \rho(\alpha)$, which shows that $\alpha\GLscGRs\delta_1$ by Proposition \ref{LoRstars in TAB} and finishes the first part of the proof.

Now assume that $e<\rho(\alpha)<\dim B$, then there exists an element $z\in B$ such that $z\notin \im\alpha$.
Using the same notations as above with this time $\Card{I}=\rho(\alpha) \geq 1$, define $\delta_2$ and $\gamma_2$ by
$$ \delta_2 = \pmap{b_{i\geq 2} & b_1 & c_k & x_j\\b_i\alpha & d & d & b_1\alpha} \quad \text{and}
\quad \gamma_2 = \pmap{b_i & c_k & x_j\\b_i\alpha & z & x_j\alpha},$$
where the set $\set{b_{i\geq 2}}$ is again possibly empty and the element $d\in B$ is defined in a similar way as before, that is, $d$ is chosen as any constant if $e=0$ and is set to $b_2\alpha$ (which then exists) otherwise.
Then clearly $\delta_2\notin Q$ and $\im\alpha=\im\delta_2$ so $\alpha\GLs\delta_2$. 
Also $\set{x_j\alpha} \subseteq \clotX[\set{b_i\alpha}]$ so $\gamma_2\in Q$.
Finally, $\rho(\gamma_2) = \Card{I}+1 \leq \rho(\delta_2\lvert_B)+\codim[A] B$ and since $\rho(\delta_2\lvert_B) = \Card{I}-1 < \rho(\gamma_2)$ we have that $\gamma_2\GLscGRs\delta_2$ from Proposition \ref{LoRstars in TAB}, finishing the proof of the second part of the lemma.

Now consider $\beta\in\TAB$ such that $e<\rho(\beta)<\aleph_0$. 
If $\beta\notin Q$, then by Lemma \ref{Reg and non-reg maps with same image}, there exists $\beta'\in Q$ such that $\im\beta'=\im\beta$ and thus, $\beta'\GLs\beta\GRs\beta$, so we can assume that $\beta\in Q$ in the first place.
Similarly, we can assume that $\rho(\alpha) \geq \rho(\beta)$ since otherwise we can exchange the role of $\alpha$ and $\beta$.
If $\rho(\alpha) = \rho(\beta)$, then $\alpha\GLscGRs\beta$ by Lemma \ref{lem maps of same rank are D* rel}, so we can assume that $\rho(\alpha) > \rho(\beta)$.
Set $m = \rho(\alpha) - \rho(\beta)$ and construct $\gamma_1, \dots, \gamma_m\in Q$ by the process described in the first part of the lemma with the following properties:
\begin{itemize}
    \item $\rho(\gamma_1) = \rho(\alpha)-1$ and $\alpha\left(\GLscGRs\right)^2\gamma_1$;
    \item $\rho(\gamma_m)=\rho(\beta)$; and
    \item $\rho(\gamma_{r+1})=\rho(\gamma_r)-1$ and $\gamma_r \left(\GLscGRs\right)^2\gamma_{r+1}$ for all $1\leq r \leq m-1$.
\end{itemize}
Then we have that $\alpha\left(\GLscGRs\right)^{2m}\gamma_m$ and $\gamma_m\GLscGRs\beta$.
Therefore, in all cases, we have that $\alpha\left(\GLscGRs\right)^n \beta$ for some integer $n$.
\qed\end{proof}

The two previous lemmas were concerned by maps of finite rank. 
The following two on the other hand, are focused on maps of infinite rank, whenever this is possible. 
We first show that if the subalgebra $\AlgB$ has a codimension smaller than its dimension, then the relation $\GLscGRs$ can only keep the higher ranks stable.

\begin{lemma}\label{GLscGRs big infinites}
Assume that $\codim[A]B < \aleph_\kappa \leq \dim B$ and let $\alpha\in\TAB$ be such that $\rho(\alpha) = \aleph_\kappa$. Then $\alpha\GLscGRs\beta$ for some $\beta\in\TAB$ if and only if $\rho(\beta) = \rho(\alpha)$.
\end{lemma}

\begin{proof}
Since one direction is already given by Lemma \ref{lem maps of same rank are D* rel}, we assume that $\alpha\GLscGRs\beta$ and we go through the cases of Proposition \ref{LoRstars in TAB}.
If $\im\alpha\cong\im\beta$, then we directly have that $\rho(\alpha) = \rho(\beta)$.
On the other hand, if $\im\alpha\cong\im (B\beta)$ then $\rho(\beta\lvert_B) =\aleph_\kappa$. 
But then $\aleph_\kappa = \rho(\beta\lvert_B)\leq \rho(\beta) \leq \rho(\beta\lvert_B) + \codim[A]B = \aleph_\kappa$, and therefore $\rho(\beta) = \aleph_\kappa = \rho(\alpha)$.
Notice that the third case of the proposition cannot occur since we would have that $\aleph_\kappa = \rho(\alpha) \leq \rho(\beta\lvert_B) +\codim[A]B$ which forces $\rho(\beta\lvert_B) = \aleph_\kappa$ and thus the condition $\rho(\beta\lvert_B)<\rho(\alpha)$ is not satisfied.
\qed\end{proof}

In another way, if the dimension and codimension of $B$ are both infinite cardinals, from a map with rank at least $\aleph_0$, we can reach maps with larger infinite rank through the relation $\GLscGRs$, as long as we do not go further than the codimension of $B$. 
This idea is given more formally in the following Lemma.

\begin{lemma}\label{GLscGRs jumping through infinites}
Assume that $\dim B\geq \aleph_0$ and $\codim[A]B=\aleph_k$ and set $M = \min\{\dim B, \codim[A] B\}$. Let $\alpha\in Q$ be such that $\aleph_0\leq \rho(\alpha) < M$. Then for all $\nu$ with $\rho(\alpha)<\nu\leq M$ there exists $\beta\in\TAB$ such that $\rho(\beta) = \nu$ and $\alpha\GLscGRs\beta$.
\end{lemma}

\begin{proof}
Let $\alpha\in Q$. Then we can write $A\alpha = B\alpha = \clotX[\set{b_k\alpha}]$,  $B=\clotX[\set{b_k}\sqcup\set{c_i}]$ and $A = \clotX[\set{b_k}\sqcup\set{c_i}\sqcup\set{a_j}]$.
By assumption on the rank of $\alpha$, we have that $\Card{K\sqcup I} = \dim B > \rho(\alpha) = \Card{K}\geq \aleph_0$, and thus it follows that $\Card{I} = \dim B > \aleph_0$.
Now, let $\nu$ be such that $\rho(\alpha)<\nu\leq M = \min\{\Card{I}, \Card{J}\}$. 
Then there exist sets $S\subseteq J$ and $S'\subseteq I$ such that $\Card{S} = \Card{S'} = \nu$, and we let $\phi\colon S\to S'$ be a bijection between them.
For all $j\in J$ we now set elements $z_j\in B$ by $z_j = c_{j\phi}$ if $j\in S$ and $z_j = c_1$ otherwise, and we define the map $\beta\in \TAB$ as:
$$ \beta = \pmap{ b_k & c_i & a_j \\ b_k\alpha & b_1\alpha & z_j }.$$
Clearly we have that $\beta\notin Q$ and $\rho(\beta) = \Card{K\sqcup S} = \Card{K} + \nu =\rho(\alpha) + \nu = \nu$.
Also $\im(B\beta) = \clotX[\set{b_k\alpha}] = \im\alpha$ and from the second case of Proposition \ref{LoRstars in TAB} we have that $\alpha\GLscGRs\beta$, which concludes the proof.
\qed\end{proof}

We now have all the tools needed to prove the characterisation of $\GDs$ that differs greatly depending on the corank of our subalgebra $\AlgB$ as is shown by the theorem below.

\begin{theorem}\label{GDs in TAB}
Let $\alpha,\beta\in \TAB$. Then $\alpha\GDs\beta$ if and only if one of the following happens:
\begin{enumerate}[label=\rm\Roman*), ref=\Roman*)]
    \item $\codim[A] B=1$ and $\rho(\alpha)=\rho(\beta)$; \label{enum:D* cond 1}
    \item $2\leq \codim[A] B < \aleph_0$ and either $e<\rho(\alpha),\rho(\beta)<\aleph_0$ or $\rho(\alpha) = \rho(\beta)$; \label{enum:D* cond 2}
    \item $\codim[A]B = \aleph_\kappa$ and either $e<\rho(\alpha),\rho(\beta)\leq \aleph_\kappa$ or $\rho(\alpha) = \rho(\beta)$. \label{enum:D* cond 3}
\end{enumerate}
\end{theorem}
\begin{proof}
Let $\alpha, \beta\in \TAB$ be such that $\alpha\GDs\beta$. Then there exists a finite sequence $\gamma_1, \dots, \gamma_{2n}$ such that 
$$\alpha\GLs\gamma_1\GRs\gamma_2\GLs\dots \GLs\gamma_{2n-1}\GRs\gamma_{2n}=\beta.$$

In order to show that the conditions of the theorem are necessary, we focus on the different situations. 
Suppose first that $\rho(\alpha)=e$. 
Then, by Proposition \ref{GLt in TAB}, we have that $\im\alpha = \im\gamma_1$ and thus $\rho(\gamma_1)=e$ and $\gamma_1\in Q$. 
From this, we use Proposition \ref{GRs in TAB} to get that $\gamma_2\in Q$ and $\ker\gamma_1 = \ker\gamma_2$. But then $\im\gamma_1\cong \im\gamma_2$ so that $\rho(\gamma_2) = e$.
By induction, this argument gives us that $\rho(\gamma_{2i-1})=\rho(\gamma_{2i}) =e$ for all $1\leq i\leq n$ and thus $\rho(\beta)=e$.
Therefore, in all settings, a map in $T_{e^+}$ can only be $\GDs$-related to another one with the same rank e. 

Similarly, assume that $\codim[A]B = 1$. If $\rho(\alpha)$ is finite, then Lemma \ref{GLscGRs under dab 1} tells us that $\rho(\gamma_2)=\rho(\alpha)$ and by induction, we obtain that $\rho(\beta) = \rho(\alpha)$. 
Furthermore, for any $\delta\in\TAB$ we can see that if $\rho(\delta\lvert_B) + \codim[A]B \geq \aleph_0$, then this forces $\rho(\delta\lvert_B)\geq \aleph_0$ and if the inequalities $\rho(\delta\lvert_B)\leq \rho(\delta)\leq \rho(\delta\lvert_B)+\codim[A]B$ are satisfied, we also get that $\rho(\delta) = \rho(\delta\lvert_B)$. 
From this fact, if $\rho(\alpha) \geq \aleph_0$ then the third case of Proposition \ref{LoRstars in TAB} cannot happen, and in the other cases of the proposition, we directly have that $ \rho(\gamma_2\lvert_B) = \rho(\gamma_2) = \rho(\alpha)$.
By induction, we conclude again that $\rho(\beta) = \rho(\alpha)$. Case \ref{enum:D* cond 1} is therefore proved.

From now on, we assume that $\rho(\alpha) >e$ and $\codim[A]B\geq 2$. 
Using Proposition \ref{LoRstars in TAB}, we can get the following inequalities between the rank of $\gamma_2$ and that of $\alpha$, depending on the case $\alpha\GLscGRs\gamma_2$ falls into:
\begin{enumerate}
\item if $\gamma_2\in Q$ and $\im\alpha \cong\im\gamma_2$, then $\rho(\alpha) = \rho(\gamma_2)$;
\item if $\gamma_2\notin Q$ and $\im (B\gamma_2) \cong \im\alpha$, then $\rho(\alpha) = \rho(\gamma_2\lvert_B)\leq \rho(\gamma_2)\leq \rho(\gamma_2\lvert_B)+\codim[A]B = \rho(\alpha)+\codim[A]B$;
\item otherwise, $\gamma_2\notin Q$ and we have that $\rho(\alpha)\leq \rho(\gamma_2\lvert_B) + \codim[A]B \leq \rho(\gamma_2) + \codim[A]B$ and also $\rho(\gamma_2\lvert_B)<\rho(\alpha)$. These two inequalities combined give us that $\rho(\gamma_2) \leq \rho(\gamma_2\lvert_B) + \codim[A]B \leq \rho(\alpha) + \codim[A]B$.
\end{enumerate}
In all cases we can see that the inequalities $\rho(\alpha) \leq \rho(\gamma_2) + \codim[A]B$ and $\rho(\gamma_2)\leq \rho(\alpha) + \codim[A]B$ always hold.
Thus, by induction on $n$, we get that 
\begin{equation}\tag{$\star$}\label{eqn ineq pf D*}
\rho(\alpha)\leq \rho(\beta) + n\cdot\codim[A]B \qquad \text{and}\qquad \rho(\beta)\leq \rho(\alpha) + n\cdot\codim[A]B.
\end{equation}

If we assume $\rho(\alpha), \codim[A]B < \aleph_0$, then we have that $\rho(\beta) \leq \rho(\alpha) + n\cdot\codim[A]B <\aleph_0$ and $\rho(\beta)>e$ by the earlier argument since $\GDs$ is symmetric.
This gives us the first part of case \ref{enum:D* cond 2}.

Similarly, if $\aleph_0\leq \rho(\alpha)\leq\codim[A]B = \aleph_\kappa$, then $\rho(\beta) \leq \rho(\alpha) + n\cdot \codim[A]B = \codim[A]B$ and $\rho(\beta) >e$ as above. 
Thus we have $e < \rho(\beta) \leq \aleph_\kappa$, which corresponds to the first part of case \ref{enum:D* cond 3}.

Lastly, if $\rho(\alpha)\geq\aleph_0$ and $\rho(\alpha) > \codim[A]B$, then the left inequality of \eqref{eqn ineq pf D*} forces $\rho(\beta)>\codim[A]B$ and then $\rho(\beta)\geq \aleph_0$.
Hence, $\rho(\alpha)\leq \rho(\beta) + n\cdot\codim[A]B=\rho(\beta) \leq \rho(\alpha)+n\cdot\codim[A]B = \rho(\alpha)$ and thus $\rho(\beta) = \rho(\alpha)$. 
This argument shows that if $2\leq \codim[A]B < \aleph_0$ and $\rho(\alpha)\geq \aleph_0$, then $\rho(\beta) = \rho(\alpha)$, which corresponds to the second part of case \ref{enum:D* cond 2}. 
Similarly, if $\codim[A]B=\aleph_\kappa$ and $\rho(\alpha)>\aleph_\kappa$, we also have that $\rho(\beta)=\rho(\alpha)$, giving us the second part of case \ref{enum:D* cond 3}.

Since all the possible values for $\codim[A]B$ and $\rho(\alpha)$ are covered in the different arguments above, we have therefore proved that the conditions stated in the theorem are necessary conditions to have $\alpha\GDs\beta$.

\medskip
Conversely, we now assume that one of conditions \ref{enum:D* cond 1}, \ref{enum:D* cond 2} or \ref{enum:D* cond 3} hold and we verify that this is sufficient to get $\alpha\GDs\beta$. 
In other words, for any $\alpha,\beta\in\TAB$, we want to show that $\alpha(\GLscGRs)^n\beta$ for some natural number $n$ whenever one of these conditions is satisfied.

Notice that if $\beta \notin Q$, then there exists $\beta'\in Q$ such that $\im\beta' = \im\beta$ by Lemma \ref{Reg and non-reg maps with same image}. This means that $\beta'\GLs\beta$ by Proposition \ref{GLt in TAB} and thus $\beta'\GLscGRs\beta$. Therefore we can assume from now on that $\beta\in Q$, as this won't change the finiteness of the sequence of relations.

We already know from Lemma \ref{lem maps of same rank are D* rel} that if $\rho(\alpha)= \rho(\beta)$, then $\alpha\GDs\beta$, which shows that the appropriate part of cases \ref{enum:D* cond 1}, \ref{enum:D* cond 2} and \ref{enum:D* cond 3} are sufficient conditions to get that the two maps are $\GDs$-related.
The only possibilities left to verify are those where the ranks of $\alpha$ and $\beta$ are different and lie in the intervals given in conditions \ref{enum:D* cond 2} and \ref{enum:D* cond 3}.

Assume that $2\leq \codim[A]B$ and that $e<\rho(\alpha), \rho(\beta)<\aleph_0$, that is, either condition \ref{enum:D* cond 2} holds, or we have condition \ref{enum:D* cond 3} with two maps of finite rank. 
Then, by invoking the third part of Lemma \ref{GLscGRs under dab>1}, we have that $\alpha(\GLscGRs)^n\beta$ and therefore $\alpha\GDs\beta$.

From now on, we assume that condition \ref{enum:D* cond 3} holds with $\codim[A]B = \aleph_\kappa$ and $e < \rho(\alpha), \rho(\beta) \leq \aleph_\kappa$. Without loss of generality, we also assume that $\rho(\alpha) < \rho(\beta)$.
If both ranks are infinite, then we can use Lemma \ref{GLscGRs jumping through infinites} with $\nu = \rho(\beta)$ (since $\dim B \geq\rho(\beta) > \rho(\alpha) \geq \aleph_0$ in that case) to get a map $\gamma\in\TAB$ such that $\alpha\GLscGRs\gamma$ and $\rho(\gamma) = \rho(\beta)$. But then $\alpha(\GLscGRs)^2\beta$ using Lemma \ref{lem maps of same rank are D* rel} and thus $\alpha\GDs\beta$.

The only situation left is when $\rho(\alpha) < \aleph_0\leq \rho(\beta)$. For this, we write $A\alpha = \clotX[\set{x_k\alpha}]$, $B = \clotX[\set{y_k}\sqcup\set{b_i}]$ and $A = \clotX[\set{y_k}\sqcup\set{b_i}\sqcup\set{a_j}]$.
Since $e<\Card{K}<\aleph_0$ we set $L = K\priv\set{1}$, so that $\Card{L} = \Card{K} - 1$ and $\set{y_k} = \set{y_1} \sqcup \set{y_\ell}$. 
By assumption, we have that $\Card{K\sqcup I} = \dim B \geq \rho(\beta)\geq\aleph_0$ so that $\Card{I}\geq \aleph_0$, and $\Card{J} = \codim[A]B = \aleph_\kappa$. 
Therefore, there exist $S\subseteq J$ and $S'\subseteq I$ such that $\Card{S} = \Card{S'} = \aleph_0$ and we let $\phi\colon S\to S'$ be a bijection between these sets.
For all $j\in J$, we now set elements $z_j\in B$ by $z_j = b_{j\phi}$ if $j\in S$ and $z_j=b_1$ otherwise. 
With this, we define $\gamma_1\in \TAB$ as:
$$ \gamma_1 = \pmap{y_1 & y_\ell & b_i & a_j\\ c & y_\ell & c & z_j},$$
where $c\in \clotX[\set{y_\ell}]$ (which necessarily exists following our assumptions on $\TAB$ after Corollary \ref{cor tab not reg in general}).
Then we have that $B\gamma_1 = \clotX[\set{y_\ell}] \subsetneq \clotX[\set{y_\ell}\sqcup\set{z_j}] = A\gamma_1$ so that $\gamma_1\notin Q$. 
Moreover, we have that 
$$ \rho(\gamma_1\lvert_B) = \Card{L} < \Card{K} = \rho(\alpha) < \aleph_0 \leq \aleph_\kappa = \rho(\gamma_1\lvert_B) + \codim[A]B,$$
and thus $\alpha\GLscGRs\gamma_1$ by the third case of Proposition \ref{LoRstars in TAB}. 
Now, either $\aleph_0 = \rho(\gamma_1) = \rho(\beta)$ and we directly get that $\gamma_1\GLscGRs\beta$ by Lemma \ref{lem maps of same rank are D* rel}, or we have that $ \aleph_0 = \rho(\gamma_1)  < \rho(\beta)$.
If the latter occurs, then we also have that $\rho(\beta) \leq \min\set{\dim B, \codim[A]B}\leq \aleph_\kappa$ by the initial assumptions. 
Thus, we can invoke Lemma \ref{GLscGRs jumping through infinites} with $\nu = \rho(\beta)$ to get a map $\gamma_2\in \TAB$ such that $\rho(\gamma_2) = \rho(\beta)$ and $\gamma_1\GLscGRs\gamma_2$ which means that $\gamma_1(\GLscGRs)^2\beta$ by Lemma \ref{lem maps of same rank are D* rel}.
Therefore, in both situations, we have that $\alpha(\GLscGRs)^3\beta$ and thus $\alpha\GDs\beta$, which finishes the proof of the characterisation of $\GDs$.
\qed\end{proof}

\begin{remark}
From the characterisation of $\GDs$ in Theorem \ref{GDs in TAB} it is easy to see that if $B$ is finite dimensional and $\codim[A]B\geq 2$, then $\GDs$ is made of only $2$ classes, namely $T_{e^+}$ and $\Tec$.
\end{remark}
The last extended Green's relations investigated here are the relations $\GJs$ and $\GJt$, described in \cite{Fountain-abundant,RSG10}.
The relation $\GJs$ on a semigroup $S$ is given by the condition that $\alpha\GJs\beta$ if and only if $J^*(\alpha)=J^*(\beta)$, where $J^*(\alpha)$ is the smallest ideal of $S$ containing $\alpha$ and that is saturated by both $\GLs$ and $\GRs$, that is, the smallest ideal containing $\alpha$ that is the union of $\GLs$-classes and $\GRs$-classes of its elements.
The relation $\GJt$ is defined similarly on $S$ using the $\sim$-relations $\GLt$ and $\GRt$, and is thus a union of $\GLt$ and $\GRt$-classes of $S$.
Note also that we have the inclusions $\GDs\subseteq \GJs$, $\GDt\subseteq\GJt$, $\GDs\subseteq\GDt$ and $\GJs\subseteq\GJt$.

Since we have determined the $\GDs$-classes, and $\GDs$ is a subset of $\GJs$, $\GDt$ and $\GJt$, the characterisation of these remaining extended Green's relations on $\TAB$ is quite straightforward to determine. 
In order to work with $\GJs$, we use the equivalence proved by Fountain in \cite{Fountain-abundant}, namely that $\beta\in J^*(\alpha)$ if and only if there exist $\gamma_0, \dots, \gamma_n\in \TAB$ and $\lambda_1,\dots, \lambda_n, \mu_1, \dots, \mu_n\in \TAB^1$ such that $\alpha=\gamma_0, \beta=\gamma_n$ and $(\gamma_i, \lambda_i\gamma_{i-1}\mu_i)\in \GDs$ for all $1\leq i\leq n$.
An equivalent characterisation for $\GJt$ in terms of $\GDt$ has been given by Ren, Shum and Guo in \cite{RSG10}.
From this we have the short proposition that follows:

\begin{proposition}\label{GJs in TAB}
In \TAB, we have that $\GDs=\GJs$.
\end{proposition}

\begin{proof}
Since we know that $\GDs\subseteq \GJs$, it remains to show the converse.
To this end, we are going to determine the $\GJs$-classes of certain cases by describing the principal $*$-ideal generated by specific elements. 
Combining these cases with the description of the $\GDs$-classes given in Theorem \ref{GDs in TAB} will finish the proof.

As a general setup, we consider $\alpha, \beta\in\TAB$ with $\beta\in J^*(\alpha)$. 
Then, there exist a finite sequence $\gamma_0, \dots, \gamma_n\in\TAB$ and some finite sets $\set{\lambda_i}, \set{\mu_i}\subseteq\TAB^1$ such that $\gamma_0 = \alpha$, $\gamma_n = \beta$ and $(\gamma_i, \lambda_i\gamma_{i-1} \mu_i)\in\GDs$ for all $1\leq i\leq n$.

Assume first that $\codim[A]B = 1$. Then, Theorem \ref{GDs in TAB} gives us that $\rho(\gamma_1) = \rho(\lambda_1\alpha\mu_1)\leq \rho(\alpha)$ and by induction, we obtain $\rho(\gamma_i)\leq \rho(\alpha)$ for all $1\leq i \leq n$, from which we have that $\rho(\beta) \leq \rho(\alpha)$. 
Hence, $J^*(\alpha) = \set{\beta\in\TAB \mid \rho(\beta)\leq \rho(\alpha)}$.
Reversing the roles of $\alpha$ and $\beta$ we have that $J^*(\beta) = \set{\alpha\in\TAB \mid \rho(\alpha)\leq\rho(\beta)}$. 
Therefore $\alpha\GJs\beta$ if and only if $J^*(\alpha)=J^*(\beta)$, which forces $\rho(\alpha) = \rho(\beta)$.

From now on, we assume that $\codim[A]B\geq 2$.
In the case where $\rho(\alpha) = e$, then we also have that $\rho(\lambda_1\alpha\mu_1)=e$. 
Since $\gamma_1\GDs\lambda_1\alpha\mu_1$, this forces $\rho(\gamma_1) = \rho(\lambda_1\alpha\mu_1)$ by condition \ref{enum:D* cond 2} of Theorem \ref{GDs in TAB}, and thus $\rho(\gamma_1)=e$. 
By induction on the $\gamma_i$'s, we get that $\rho(\beta) = e = \rho(\alpha)$. 
Therefore $J^*(\alpha) = T_{e^+}$, from which we have that if $\alpha\GJs\beta$ with $\rho(\alpha)=e$, then $\rho(\beta) = \rho(\alpha)$.

Consider the case when $\codim[A]B = \aleph_\kappa$ and suppose that $\rho(\beta) > \codim[A]B$. 
Then, from the fact that $\beta = \gamma_n \GDs \lambda_n\gamma_{n-1}\mu_n$, we necessarily have that $\rho(\beta) = \rho(\lambda_n\gamma_{n-1}\mu_n)$ from \ref{enum:D* cond 3} of Theorem \ref{GDs in TAB} and therefore $\rho(\gamma_{n-1}) \geq \rho(\beta) > \codim[A]B$. 
By reverse induction on $i$ from $n-1$ to $1$, we get that $\rho(\gamma_{i-1}) \geq \rho(\gamma_i) > \codim[A]B$ for all $i$ and thus $\rho(\alpha) \geq \rho(\beta) > \codim[A]B$. 
This shows that a map $\alpha$ can only contain in its $J^*$-ideal a map $\beta$ with $\rho(\beta)>\codim[A]B \geq \aleph_0$ if $\rho(\alpha) \geq \rho(\beta)$ and $\rho(\alpha)>\codim[A]B$ in the first place. 
Together with the reverse statement, we conclude that if $\rho(\alpha) > \codim[A]B\geq \aleph_\kappa$, then $\alpha\GJs\beta$ implies $\rho(\alpha) = \rho(\beta)$.

Using a similar argument when $\codim[A]B < \aleph_0$, we also have that if $\beta\in J^*(\alpha)$ with $\rho(\beta) \geq \aleph_0$, then $\rho(\alpha) \geq \rho(\beta)$.
Therefore, this case gives us that if $\rho(\alpha)\geq\aleph_0 > \codim[A]B$ and $\alpha\GJs\beta$, then $\rho(\alpha) =\rho(\beta)$.

In all of the above cases, we can see that if two maps $\alpha$ and $\beta$ are $\GJs$-related, then they have the same rank and thus they are $\GDs$-related by Lemma \ref{lem maps of same rank are D* rel}. 
The remaining cases to consider are when we either have that $2\leq \codim[A]B<\aleph_0$ and $e<\rho(\alpha), \rho(\beta)<\aleph_0$, or we have that $\codim[A]B = \aleph_\kappa$ and $e<\rho(\alpha), \rho(\beta)\leq \aleph_\kappa$.
However, these cases are already given to be $\GDs$-classes as the first part of conditions \ref{enum:D* cond 2} and \ref{enum:D* cond 3} in Theorem \ref{GDs in TAB}, and are thus also $\GJs$-classes, which finishes to show that $\GJs = \GDs$.
\qed\end{proof}

Finally, for the relations $\GDt$ and $\GJt$ there are always only two classes, since the corank of $\AlgB$ has no impact in that situation as given by this last proposition.

\begin{proposition}
In \TAB, the only $\GDt$ and $\GJt$ classes are $T_{e^+}$ and $\Tec$.
\end{proposition}
\begin{proof}
Let $\alpha, \beta\in \TAB$ be such that $\rho(\alpha) = \rho(\beta) = e$.
Then $\alpha\GDs\beta$ by Lemma \ref{lem maps of same rank are D* rel}, and thus $\alpha\GDt\beta$ since $\GDs\subseteq \GDt$.
On the other hand, assume that the rank of both $\alpha$ and $\beta$ is strictly greater than $e$.
If $\alpha\in Q$ then, by Lemma \ref{Reg and non-reg maps with same image}, there exists $\alpha'\in \Qc$ such that $\im\alpha = \im\alpha'$, while if $\alpha\in \Qc$ in the first place, we simply set $\alpha' = \alpha$. 
In both cases, we have that $\alpha\GLt\alpha'$ and similarly, $\beta\GLt\beta'$ for some $\beta'\in \Qc$.
Then, by Proposition \ref{GRt in TAB}, we get that $\alpha'\GRt\beta'$, and so $\alpha\GLt\circ\GRt\circ\GLt\beta$. Therefore $\alpha\GDt\beta$ for any $\alpha,\beta\in \Tec$.

For the converse, notice first that if $\gamma\in\TAB$ is such that $\rho(\gamma)=e$, then for any $\delta\in\TAB$ we have that $\rho(\delta)=e$ whenever $\gamma\GRt\delta$ or $\gamma\GLt\delta$.
Indeed, if $\gamma\GRt\delta$ then, by Proposition \ref{GRt in TAB} together with the fact that $\gamma\in Q$, we get that $\delta\in Q$ and $\ker\delta = \ker\gamma$. 
Consequently, $\im\delta\cong\im\gamma$ and so $\rho(\delta)=\rho(\gamma) = e$.
Similarly, if $\gamma\GLt\delta$, then Proposition \ref{GLt in TAB} gives us that $\im\delta = \im\gamma$ and thus $\rho(\delta)=\rho(\gamma)=e$, proving the claim.
Now consider $\alpha, \beta\in\TAB$ such that $\alpha\GDt\beta$.
Then there exists a finite sequence of composition of $\GLt$ and $\GRt$ relating $\alpha$ to $\beta$. 
If $\rho(\alpha)=e$, then the arguments exposed above gives us that all maps in that sequence have rank $e$, and thus $\rho(\beta) = e$. 
Similarly, if $\rho(\alpha)\neq e$, then by symmetry of the arguments, we necessarily get that $\rho(\beta)\neq e$, which concludes the proof of the characterisation of $\GDt$.

\medskip
Invoking arguments similar to those used in the proof of Proposition \ref{GJs in TAB} together with the newfound characterisation of $\GDt$, one can show that a map in $T_{e^+}$ can only be $\GJt$-related to another map in $T_{e^+}$. 
Since $\GDt\subseteq\GJt$ and $\GDt$ only has two classes, it follows that these equivalence relations are in fact equal.
\qed\end{proof}

\section*{Acknowledgements}
I wish to thank the referee for their many suggestions to improve this paper, and for a notable remark which, together with a question from Carl-Fredrik Nyberg-Brodda, allowed me to strengthen the results of the last section.
I am also grateful to Pr. Victoria Gould for her generous supervision and her keen eye on chasing the Francisation of my writing.

\end{document}